\documentclass[a4paper,oneside,10pt]{article}

\usepackage[utf8]{inputenc}
\usepackage[a4paper, height = 24cm]{geometry}
\usepackage{pgfplots}
\usetikzlibrary{arrows.meta}

\usepackage{nicefrac}
\usepackage{amsmath,amssymb}
\usepackage{dsfont}

\usepackage{array}
\usepackage{booktabs}
\usepackage{tikz}
\usepackage{graphicx}
\usepackage{bbm}
\usepackage{comment}

\usepackage{url}
\usepackage{nameref}
\usepackage[colorlinks, linkcolor = black, citecolor = black, filecolor = black, urlcolor = blue]{hyperref}

\usepackage{ifthen}
\usepackage{natbib}
\newcommand*{\doi}[1]{doi: \href{https://doi.org/#1}{\begingroup\urlstyle{rm}\nolinkurl{#1}\endgroup}}
\usepackage{color}
\usepackage{epstopdf}
\usepackage{caption}
\usepackage{subcaption}
\usepackage{arydshln}
\usepackage{longtable}
\usepackage{marvosym}
\usepackage{cases}
\usepackage{enumitem}

\usepackage{amsthm}

\theoremstyle{plain}
\newtheorem{theorem}{Theorem}
\newtheorem{lemma}[theorem]{Lemma}
\newtheorem{proposition}[theorem]{Proposition}

\theoremstyle{definition}

\newtheorem{example}[theorem]{Example}
\newtheorem{remark}[theorem]{Remark}
\newtheorem*{remark*}{Remark}
\newtheorem{assumption}{Assumption}

\newcommand{\rd}{\mathrm{d}}

\newcommand{\mc}{\mathcal}
\newcommand{\dx}{\mathrm{d}}
\newcommand{\e}{\varepsilon}
\newcommand{\bs}[1]{\boldsymbol{#1}}

\newcommand{\dd}{\mathrm{d}}

\newcommand{\expec}{{\mathbb{E}}}
\newcommand{\prob}{{\mathbb{P}}}

\newcommand{\sign}{{\operatorname{sign}}}

\newcommand{\floor}[1]{\left\lfloor#1\right\rfloor}
\newcommand{\ceil}[1]{\left\lceil#1\right\rceil}
\newcommand{\abs}[1]{\lvert #1 \rvert}

\newcommand{\card}{\operatorname{card}}

\newcommand{\Op}[1]{{\operatorname{\mathcal{O}_\mathbb{P}}\left(#1\right)}}

\newcommand{\Oop}[1]{{\operatorname{\mathcal{O}}\left(#1\right)}}

\newcommand{\norm}[1]{{\lVert#1\rVert}}

\newcommand{\normHS}[1]{{\lVert#1\rVert_{\mathrm{HS}}}}

\newcommand{\normLC}[1]{{\lVert#1\rVert_{\mathrm{L_2\to C}}}}

\newcommand{\norminf}[1]{{\lVert#1\rVert_{\infty}}}

\newcommand{\normb}[1]{{\big\lVert#1\big\rVert}}

\newcommand{\absb}[1]{\big|#1\big|}

\newcommand{\cd}{{c_\mathrm{des}}}

\newcommand{\R}{{\mathbb{R}}}
\newcommand{\N}{{\mathbb{N}}}

\newcommand*{\defeq}{\mathrel{\vcenter{\baselineskip0.5ex \lineskiplimit0pt
			\hbox{\scriptsize.}\hbox{\scriptsize.}}}=}
\newcommand*{\defeql}{ = \mathrel{\vcenter{\baselineskip0.5ex   	\lineskiplimit0pt
			\hbox{\scriptsize.}\hbox{\scriptsize.}}}}

\usepackage[colorinlistoftodos,textsize=tiny]{todonotes}
\newcommand{\Comments}{0}

\ifnum\Comments=1
\fi

\newcommand{\wjk}[3]{w_{j,k}({#1},{#2};{#3})}

\newcommand{\Zclass}[0]{\mc P(\gamma)}
\newcommand{\Zclassk}[0]{\mc P(\gamma;k)}

\newcommand{\covest}[3]{\hat \Gamma_{n}({#1}, {#2};{#3})}

\newcommand{\hgam}{\hat \Gamma}

\newcommand{\phikest}[2]{\hat\phi_{k;n}({{#1};{#2}})}
\newcommand{\lamkest}[1]{\hat\lambda_{k;n}({{#1}})}

\newcommand{\Cmax}{C_1}
\newcommand{\Clip}{C_2}
\newcommand{\Ccard}{C_3}

\usepackage[colorinlistoftodos,textsize=tiny]{todonotes}

\ifnum\Comments=1
\fi

\definecolor{revcolor}{rgb}{0,0,0}
\newcommand{\rev}[1]{{\color{revcolor}#1}}
\newenvironment{revblock}{\par\color{revcolor}}{\par}
\definecolor{draftcolor}{rgb}{0.0,0.45,0.25}

\title{Transferring supremum-norm rates and weak convergence of covariance kernel estimators to functional principal components }
\author{Hajo Holzmann and Kevin Wilk \footnote{Corresponding author. Prof. Dr. Hajo Holzmann, Department of Mathematics and Computer Science, Philipps-Universit\"at Marburg, Hans-Meerweinstr., 35043 Marburg, Germany}\\
    \small Department of Mathematics and Computer Science\\
    \small Philipps-Universit\"at Marburg\\
    \small \{holzmann, wilk\}@mathematik.uni-marburg.de }

\date{}

\begin{document}

\maketitle

\begin{abstract}
We show that $L_2$-perturbation theory can be used to transfer rates of convergence in the supremum norm as well as weak convergence in the space of continuous functions from covariance kernel estimators to the associated functional \rev{principal} components. As an application we obtain optimal rates of convergence in sup-norm, including minimax-lower bounds,  as well as asymptotic normality for estimating the functional principal components in a discrete observational model with errors under fixed, synchronous design. The sparse to  dense transition which has previously been observed for mean function and covariance kernel estimators also applies here. Surprisingly, eigenvalue estimation exhibits a discretization-dominated regime under sparse designs, too.
Our results further apply to estimators of cross-covariance and long-run covariance kernels, as well as to covariance kernels of derivative processes. We also present results of numerical experiments in which we use the Nyström method to compute functional principal components and eigenvalues, and give an empirical illustration to series of daily temperature curves.
\end{abstract}

\vspace{3mm}


\section{Introduction}

In functional data analysis for second order stochastic processes\rev{,} the eigenfunctions of the covariance operator, that is the functional principal component (FPC) basis functions, are a major inferential tool \citep{wang2016functional, ramsay2005functional}. Due to the Karhunen-Loève expansion of the process, the FPCs associated to the first few largest eigenvalues are the dominant modes of variation of the process, and hence have exploratory and scientific interpretations. Further the FPC scores, that is the inner products of FPCs and the observed process serve for dimension reduction in downstream tasks such as classification \citep{DaiMueller2027}. Here a large number of FPC scores, potentially even growing with the sample size, may be required.

For FPC estimates derived from the empirical covariance operator of a  continuously observed process,   \citet{hall2009theory, HaaMuellerWang2006} provide asymptotic expansions. For discretely observed processes, rates of convergence for the FPCs are obtained in \citet{HaaMuellerWang2006, ZhouWeiYao2025, li2010uniform, cai2010nonparametric, belhakem2025minimax, bourarach2026minimax, cai2026minimax}, both under asynchronous random sampling  as well as under synchronous, fixed sampling. \rev{In particular, under synchronous design, minimax rates in $L_2$ for FPCs are obtained in \citet{bourarach2026minimax} for covariance kernels of general H\"older smoothness, see also  \citet{belhakem2025minimax} for low Hölder smoothness of order  $\alpha \in (0,1]$.}
These results are typically derived by combining covariance kernel estimation with Hilbert-space operator perturbation arguments \citep[Chapter 5]{hsing2015theoretical}.

 But supremum norm error bounds and weak convergence in the space of continuous functions $C[0,1]$ are not naturally tied to the Hilbert space structure of $L_2$. The supremum norm arguably corresponds more closely to the visualization of the estimation error. \rev{
 Interpretations of FPCs at particular locations, such as contrasts between certain times of the day, also require uniform control of the estimation error. Further,} weak convergence in the separable Banach space $C[0,1]$ of continuous function equipped with the supremum norm \citep{billingsley1999convergence} forms the basis for the construction of uniform confidence bands \citep{liebl2023fast, dette2020functional}. Thus, both are of interest for analyzing more closely the FPCs corresponding to the largest eigenvalues.
In this paper we show how $L_2$-perturbation theory can still be used here. The key observation is that covariance operators generated by continuous kernels act continuously from $L_2$ into $C[0,1]$, with operator norm controlled by the supremum norm of the kernel. Together with the eigenfunction-eigenvalue equation of FPCs this allows $L_2$-perturbation expansions to be combined with supremum norm bounds for covariance kernel estimators.
Our results also apply to estimates of symmetric, continuous, not necessarily positive semidefinite kernels of integral operators such as symmetrized cross-covariance kernels.

As an application, we establish minimax-optimal rates and asymptotic normality in $C[0,1]$ for estimating FPCs and eigenvalues under synchronous fixed designs with observational errors. The rates exhibit the same sparse-to-dense phase transition previously observed for mean function and covariance kernel estimation \citep{berger2024dense, berger2024optimal}. To derive matching lower bounds for eigenfunctions, we introduce a semimetric that accounts for the inherent sign indeterminacy of eigenfunctions corresponding to simple eigenvalues. Interestingly, eigenvalue estimation also exhibits a discretization-dominated regime under sparse designs.
Further, for differentiable processes we also indicate how the  upper bounds extend to estimating FPCs corresponding to the covariance kernel of the derivative process.

The paper is structured as follows. After recalling basic operator and spectral theory for integral operators with symmetric kernels in Section \ref{sec:basicfa}, in Section \ref{sec:perturbationtheory} we give a general upper bound in sup-norm for estimates of the FPCs for simple eigenvalues in terms of the sup-norm of the associated covariance kernel estimate. Section \ref{sec:weakconv} then shows how to leverage weak convergence in the space of continuous functions to FPCs. Section \ref{sec:ratesasy} contains minimax-optimal rates, including lower bounds, for estimating FPCs in case of synchronous, discrete design under the supremum norm. Finally, a numerical implementation using the Nyström method together with simulation results is  presented in Section \ref{sec:numerics}. Some additional technical derivations are deferred to an appendix.

\section{From covariance kernel estimates to functional principal components: Upper bounds and weak convergence} \label{sec:model:estimators}

\subsection{Covariance kernels and functional principal components}\label{sec:basicfa}

\rev{Let} $Z = (Z(x))_{x \in [0,1]}$ be a centered, square-integrable \rev{stochastic} process
with covariance kernel
$$\Gamma(x,y) = \expec[Z(x)\, Z(y)].$$
We assume that $\Gamma$ is continuous, which is the case if and only if $Z$ is continuous in quadratic mean.
The integral operator associated with $\Gamma$ is denoted by $T_\Gamma$.

Let us briefly discuss properties of integral operators with symmetric, square-integrable kernels $G\colon[0,1]^2 \to \R$, see e.g.~\citet[Chapter 4]{hsing2015theoretical} or \citet[Section 2]{horvath2012inference}.
The integral operator associated to $G$ is
\[ (T_G\, f)(x) = \int_0^1\, G(x,y)\, f(y)\,  \dx y, \qquad f \in L_2[0,1].\]
As a mapping $T_G\colon L_2[0,1] \to L_2[0,1]$, $T_G$ is a self-adjoint  Hilbert-Schmidt operator, and in particular compact. For the $L_2$-operator norm $\|T_G\|_{L_2}$ and the Hilbert-Schmidt norm $\normHS{T_G}$ we have that
\begin{equation}\label{eq:l2operator}
\|T_G\|_{L_2}  \leq \normHS{T_G} = \| G \|_{L_2([0,1]^2)}.
\end{equation}
\rev{Here $\|f\|_{L_2} = (\int_0^1 f(x)^2\,\dx x)^{1/2}$ is the norm of $L_2[0,1]$, $\|T\|_{L_2} = \sup_{\|f\|_{L_2}=1}\|Tf\|_{L_2}$ is the operator norm of a bounded linear operator $T$ on $L_2[0,1]$, $\normHS{T} = (\sum_{j}\|Te_j\|_{L_2}^2)^{1/2}$ is its Hilbert--Schmidt norm, where $(e_j)_j$ is an arbitrary orthonormal basis of $L_2[0,1]$, and $\| G \|_{L_2([0,1]^2)} = (\int_0^1\int_0^1 G(x,y)^2\,\dx x\, \dx y)^{1/2}$. We denote by $L_{2,s}([0,1]^2)$ the closed linear subspace of symmetric kernels in $L_2([0,1]^2)$, and by $C_s([0,1]^2)$ the closed linear subspace of symmetric kernels in the space $C([0,1]^2)$ of continuous functions on $[0,1]^2$ with the supremum norm.}

From the spectral theorem for compact symmetric operators we obtain an orthonormal basis $(\phi_{k,G})_{k \in \N}$ of eigenfunctions with associated eigenvalues $\lambda_{k,G}$, so that
\begin{equation}\label{eq:eigenrel}
T_G\,\phi_{k,G} = \lambda_{k,G}\, \phi_{k,G}.
\end{equation}
Every eigenvalue except possibly for $0$ has finite multiplicity, note that we repeat the eigenvalues in \eqref{eq:eigenrel} with their multiplicities.
The eigenvalue $\lambda_{k,G}$ is simple by definition if the associated eigenspace is one-dimensional. In this case the eigenfunction $\phi_{k,G}$ is uniquely determined up to sign, and
\begin{equation}\label{eq:spectralgap}
\delta_{k,G} = \min_{j \not= k} |\lambda_{k,G} - \lambda_{j,G}|>0.
\end{equation}
We call $\delta_{k,G}$ the spectral gap associated to the simple eigenvalue $\lambda_{k,G}$.

Now suppose that the kernel $G$ is continuous (and symmetric). Then, importantly for us, the operator $T_G$ maps $L_2[0,1]$ into $C[0,1]$. Further, as a linear  mapping
$ T_G\colon L_2[0,1] \to C[0,1] $, $T_G$ is also bounded (even compact) with mixed operator norm
\begin{equation}\label{eq:operatornormmixed}
\normLC{T_G} : = \sup_{f \in L_2[0,1],\, \|f\|_{L_2}=1}\, \| T_G (f)\|_{\infty} = \sup_{x \in [0,1]}\, \| G(x,\cdot)\|_{L_2[0,1]} \leq \| G \|_{\infty},
\end{equation}
where $\norm{\,\cdot\,}_\infty$ denotes the supremum norm, $\norminf{f} = \sup_{x}|f(x)|$ for continuous $f$, with the domain $[0,1]^2$ or $[0,1]$ being  clear from the context. \rev{For  \eqref{eq:operatornormmixed}, note that for $f \in L_2[0,1]$ and $x \in [0,1]$, the Cauchy--Schwarz inequality gives $|(T_Gf)(x)| = |\langle G(x,\cdot), f\rangle| \leq \|G(x,\cdot)\|_{L_2[0,1]}\, \|f\|_{L_2}$, with equality for $f = G(x,\cdot)/\|G(x,\cdot)\|_{L_2[0,1]}$ if $G(x,\cdot)\neq 0$.}
 In particular, all eigenfunctions corresponding to non-zero eigenvalues are continuous:
\begin{equation}\label{eq:eigenfctexpress}
 \phi_{k,G} = \frac{T_G\, \phi_{k,G}}{\lambda_{k,G}} \in C[0,1].
\end{equation}
Coming back to the covariance kernel $\Gamma$, in the following the eigenvalues $\lambda_{k,\Gamma}$ and eigenfunctions, also called principal component basis functions $\phi_{k,\Gamma}$ of $T_\Gamma$ will simply be denoted by $\lambda_k$ and $\phi_k$. Since $T_\Gamma$ is positive \rev{ semidefinite, that is, $\langle T_\Gamma f, f\rangle \geq 0$ for all $f\in L_2[0,1]$}, all eigenvalues are non-negative, which allows to order them decreasingly as
\begin{equation}
\lambda_1 \geq \lambda_2 \geq \ldots \geq 0,
\end{equation}
Then in \eqref{eq:spectralgap},
\begin{equation}\label{eq:specgapcovkern}
\delta_{k,\Gamma}= \delta_k = \min\big(\lambda_{k-1} - \lambda_k, \lambda_{k} - \lambda_{k+1}\big), \quad k \geq 2\quad \text{and}\quad \delta_1 = \lambda_{1} - \lambda_{2}.
\end{equation}
and $\lambda_k$ is simple if and only if $\delta_k>0$. \rev{If $T_\Gamma$ has infinitely many positive eigenvalues, the $\phi_k$ in this ordering are complemented by an orthonormal basis of the null space of $T_\Gamma$. Thus, sums such as $\sum_{j \not = k}\, a_j \phi_j$, e.g.~\eqref{eq:asympcent}, include the basis of the null space.}

\subsection{Perturbation theory and \texorpdfstring{\rev{an}}{an} upper bound in supremum norm}\label{sec:perturbationtheory}

\rev{Let $\hgam(x,y)$ be another symmetric and continuous kernel on $[0,1]^2$, which need not be positive semidefinite; the notation suggests that later $\hgam$ will be an estimator of the covariance kernel $\Gamma$.} We denote the associated integral operator by $T_{\hgam} = \hat T$.
\rev{Since $\hat T$ may have negative eigenvalues, we use the Courant--Fischer characterization $\hat\lambda_j = \sup_{V} \inf_{f\in V,\, \|f\|_{L_2}=1} \langle \hat T f, f\rangle$ of positive  eigenvalues \citep[Theorem 9.2.4]{birman1987spectral}, where the supremum is taken over all $j$-dimensional linear subspaces $V$ of $L_2[0,1]$. Thus $\hat\lambda_j$ is the $j$-th largest eigenvalue of $\hat T$, repeated with its multiplicity, in case $\hat\lambda_j>0$, that is if $\hat T$ has at least $j$ positive eigenvalues, otherwise $\hat\lambda_j = 0$.  With this convention, $\hat \lambda_1 \geq \hat \lambda_2 \geq \ldots \geq 0$, and Weyl's inequality \eqref{eq:spectraleigenvalues} below also holds for kernels which are not positive semidefinite, and for positive semidefinite kernels such as $\Gamma$ it agrees with the ordering of the eigenvalues used above. If $\hat\lambda_j>0$ let $\hat \phi_j$ be an associated normalized eigenfunction, and set $\hat\phi_j = 0$ otherwise. In the asymptotic results below, $\hat\lambda_k>0$ holds with probability tending to one whenever $\lambda_k>0$.}
Again, from \eqref{eq:eigenfctexpress} applied to $\hat T$, eigenfunctions $\hat \phi_k$ for eigenvalues $\hat \lambda_k>0$ are continuous.

Suppose that $\delta_k>0$ (given in \eqref{eq:specgapcovkern}), so that $\lambda_k$ is simple and $\phi_k$ is unique up to sign. To take into account the potentially opposite orientation of $\hat \phi_k$, set
\begin{equation}\label{eq:signcoefficient}
 c_k = \sign\big(\langle \hat \phi_k, \phi_k\rangle\big)\end{equation}
and consider the difference $\hat \phi_k - c_k\, \phi_k$.
Our arguments for analyzing estimates of principal component basis functions in the supremum norm are based on the following simple expansion:
If $\lambda_k, \hat \lambda_k  >0$  then
\begin{align}
    \hat \phi_k - c_k\, \phi_k & = \frac{\hat T\, \hat \phi_k}{\hat \lambda_k} - \frac{T_\Gamma\, c_k\, \phi_k}{\lambda_k}\nonumber\\
    & = \frac{\big(\hat T-T_\Gamma\big)\, \hat \phi_k}{\hat \lambda_k} - \frac{\hat \lambda_k - \lambda_k}{\hat \lambda_k\, \lambda_k}\, T_\Gamma\, \hat \phi_k + \frac{1}{\lambda_k}\, T_\Gamma\, \big(\hat \phi_k -c_k\, \phi_k\big). \label{eq:basicexpansion}
\end{align}
Using $\| \hat \phi_k \|_{L_2}=1$ we obtain
\begin{align*}\big\| \hat \phi_k - c_k\, \phi_k\big\|_\infty \leq &
     \frac{\normLC{\hat T-T_\Gamma}  }{\hat \lambda_k} \,
    + \, \normLC{ T_\Gamma}\, \Big(\frac{|\lambda_k - \hat \lambda_k|}{\hat \lambda_k\, \lambda_k}  + \frac{1}{\lambda_k}\,  \big\|\hat \phi_k - c_k\,  \phi_k\big\|_{L_2}\Big).
\end{align*}
A standard result in perturbation theory, \rev{Weyl's} inequality, allows to bound the difference of eigenvalues of $T_\Gamma$ and $\hat T$ in terms of the $L_2$-operator distance. \rev{For compact self-adjoint operators that are not necessarily positive semidefinite and with the above convention for $\hat\lambda_j$ it states that
\begin{align}
    \sup_{j \geq 1} \big| \hat \lambda_j - \lambda_j\big| & \leq \big\|\hat T - T_\Gamma \big \|_{L_2} , \label{eq:spectraleigenvalues}
\end{align}
see \citet[Theorem 9.2.8]{birman1987spectral}}.
Using \eqref{eq:spectraleigenvalues} and the  bounds  \eqref{eq:l2operator} and \eqref{eq:operatornormmixed} for the operator norms $\big\|\hat T - T_\Gamma \big \|_{L_2}$ and $\normLC{\hat T-T_\Gamma}$  we obtain
\begin{align}\label{eq:generalupperbound}
    \big\| \hat \phi_k - c_k\, \phi_k\big\|_\infty \leq & \frac{\big\|\hgam - \Gamma \big \|_\infty}{\hat \lambda_k}\, \Big(1 +  \frac{\big\|\Gamma \big \|_\infty}{\lambda_k}\Big) + \frac{\big\|\Gamma \big \|_\infty}{\lambda_k}\,  \big\|\hat \phi_k - c_k\,  \phi_k\big\|_{L_2}.
\end{align}
\rev{If $\lambda_k$ is a simple eigenvalue, that is, $\delta_k>0$, a} simple but useful bound for $\big\|\hat \phi_k - c_k\,  \phi_k\big\|_{L_2}$ is the sine-$\theta$-theorem (in its functional version), see e.g.~\citet[Lemma 2.3]{horvath2012inference} or \citet[Lemma 4.3]{bosq2000linear}, which states that
\begin{align}
    \| \hat \phi_k - c_k\, \phi_k\|_{L_2} & \leq \frac{2\, \sqrt 2}{\delta_k} \,  \big\|\hat T - T_\Gamma \big \|_{L_2} ,\label{eq:spectraleigenfct}
\end{align}
where $\delta_k$ is defined in \eqref{eq:specgapcovkern}.
We obtain the following overall bound.
\begin{proposition}\label{lem:supnormbound}
    Let $\Gamma$ be a continuous covariance kernel on $[0,1]^2$ with eigenfunctions $\phi_j$ and associated eigenvalues $\lambda_j$, ordered decreasingly, and let $\hgam$ be a continuous symmetric integral kernel on $[0,1]^2$ with eigenfunctions $\hat \phi_j$ and  eigenvalues $\hat \lambda_j$ (again ordered decreasingly). For given $k$, suppose that $\lambda_k>0$ is a simple eigenvalue, so that $\delta_k>0$, and that $\hat \lambda_k>0$. Then
    \begin{align}
    \big\| \hat \phi_k - c_k\, \phi_k\big\|_\infty \leq & \big\|\hgam - \Gamma \big \|_\infty\, \Big(\frac1{\hat \lambda_k} +  \frac{\big\|\Gamma \big \|_\infty}{\hat \lambda_k\, \lambda_k} + \frac{2\, \sqrt 2 \, \big\|\Gamma \big \|_\infty}{\delta_k\, \lambda_k}\Big), \label{eq:supboundeigenfct}
\end{align}
    where $c_k$ is defined in \eqref{eq:signcoefficient}.
\end{proposition}
\begin{remark}[Optimality]
    The bound \eqref{eq:supboundeigenfct} allows to leverage sup-norm bounds of covariance kernel estimates to sup-norm bounds of the associated principal component basis functions. Similar arguments are used e.g.~in the proof of \citet[Theorem 3.6, (c)]{li2010uniform}.

    While bounding the estimation error of the one-dimensional function $\phi_k$ by that of the two-dimensional function $\Gamma$ may appear \rev{suboptimal} at first glance, the covariance kernel $\Gamma$ can be estimated with one-dimensional rates in various settings, see e.g.~\citet{berger2024optimal, cai2010nonparametric, cai2026minimax}.
    In Section \ref{sec:ratesasy} we show in detail that the bound \eqref{eq:supboundeigenfct} gives optimal rates for principal component basis function estimation in a discrete observational model with synchronous design points.

    The bound \eqref{eq:supboundeigenfct} is intended and useful for fixed $k$, and does not capture the effect of the eigenvalues optimally. See the discussion in \citet{ZhouWeiYao2025} \rev{and \citet{bourarach2026minimax}}. But in applications, arguably the very first eigenfunctions corresponding to $k=1$ and potentially $k=2$ are of particular practical interest.

\end{remark}

\subsection{Transferring weak convergence}\label{sec:weakconv}

Let us now show how to transfer weak convergence of covariance kernel estimates in the space of continuous functions $C([0,1]^2)$ to that of the associated principal component basis functions in $C[0,1]$.

\rev{Let $\hgam_n$ be estimators of $\Gamma$ based on a sample of size $n$, which are symmetric and continuous kernels as in Section~\ref{sec:perturbationtheory}, and assume that}
\begin{equation}\label{eq:weakconvcovkernel}
    \sqrt n \big(\hgam_n - \Gamma \big) \,  \stackrel{D}\rightarrow \, \mc G \qquad \text{in}\quad C([0,1]^2).
\end{equation}
Here \rev{$\stackrel{D}\rightarrow$ denotes convergence in distribution of random elements of the separable Banach space $C([0,1]^2)$ with the supremum norm \citep{billingsley1999convergence}, and} $\mc G$ is a \rev{random element of} $C([0,1]^2)$, typically centered Gaussian\rev{, which takes its values in $C_s([0,1]^2)$ since the $\hgam_n$ are symmetric}.
We denote the associated $k$-th principal component basis functions and eigenvalues of $\hgam_n$ by $\hat \phi_{k;n}$ and $\hat \lambda_{k;n}$. \rev{By \eqref{eq:weakconvcovkernel}, $\|\hgam_n - \Gamma\|_\infty = \mathcal O_{\mathbb P}(n^{-1/2})$, and hence $\hat\lambda_{k;n}\to\lambda_k$ in probability by \eqref{eq:l2operator} and Weyl's inequality \eqref{eq:spectraleigenvalues}. Thus if $\lambda_k>0$, then $\prob(\hat\lambda_{k;n}>0)\to 1$, so that $\hat\phi_{k;n}$ is an eigenfunction of $\hat T_n$ with probability tending to one.}\begin{theorem}\label{th:asympnorm}
      For given $k$, suppose that $\lambda_k>0$ is a simple eigenvalue of $\Gamma$, so that $\delta_k>0$. Further assume that \eqref{eq:weakconvcovkernel} holds, and that $\phi_k$, $\hat \phi_{k;n}$ are chosen such that $\langle \rev{\hat \phi_{k;n}}, \phi_k\rangle\geq 0$. Then
\begin{align}
\sqrt n\, \big( \hat \phi_{k;n} -  \phi_k \big)\,  \stackrel{D}\rightarrow \,  \sum_{j \not= k} \frac{\langle T_{\mc G} \phi_k,\phi_j \rangle}{\lambda_k - \lambda_j} \, \phi_j\qquad \text{in}\quad C[0,1],\label{eq:asympcent}
\end{align}

    where $\mc G$ is from \eqref{eq:weakconvcovkernel} and $T_{\mc G}$ is the (random) integral operator associated with the bivariate kernel $\mc G$.
\end{theorem}
\begin{revblock}
\begin{remark}[On the proof of Theorem~\ref{th:asympnorm}]\label{rem:proofthasympnorm}
For $G\in C_s([0,1]^2)$  let $\Psi_k(G)$ be the $k$-th eigenvalue of $T_G$ in the convention of Section~\ref{sec:perturbationtheory}, that is, the $k$-th largest eigenvalue of $T_G$ if $T_G$ has at least $k$ positive eigenvalues, and $0$ otherwise. If $\|G-\Gamma\|_\infty\leq\delta_k/4$, then by Lemma~\ref{lem:eigenexpansion} in Section~\ref{sec:proofasympnorm}, $\Psi_k(G)>0$ is a simple eigenvalue of $T_G$, and we let $\Phi_k(G)$ be the associated normalized eigenfunction with sign fixed by $\langle\Phi_k(G),\phi_k\rangle>0$. Otherwise we set $\Phi_k(G) = \phi_k$; this choice is arbitrary and plays no role in the following. For the proof we show that under the assumptions of the theorem, the map $G\mapsto\big(\Psi_k(G),\Phi_k(G)\big)$ is Fr\'echet differentiable at $\Gamma$ as a mapping from  $C_s([0,1]^2)$, a Banach space with the sup-norm, to $\R \times C[0,1]$, and that the weak convergence in \eqref{eq:weakconvcovkernel} occurs also in $C_s([0,1]^2)$.

To highlight how the weak convergence in $L_2$,
\begin{align}
\sqrt n\, \big( \hat \phi_{k;n} -  \phi_k \big)\,  \stackrel{D}\rightarrow \,  \sum_{j \not= k} \frac{\langle T_{\mc G} \phi_k,\phi_j \rangle}{\lambda_k - \lambda_j} \, \phi_j\qquad \text{in}\quad L_2[0,1] \label{eq:weakinL2}
\end{align}
is transferred to weak convergence in $C[0,1]$ using a deterministic version of the expansion \eqref{eq:basicexpansion}, we show the Fr\'echet differentiability in two steps:

\textit{Step 1:} The map $G\mapsto\big(G,\Psi_k(G),\Phi_k(G)\big)$ from $C_s([0,1]^2)$ into $C_s([0,1]^2)\times\R\times L_2[0,1]$ is Fr\'echet differentiable at $\Gamma$, with derivative
$$\Delta\mapsto\Big(\Delta,\,\langle T_\Delta\phi_k,\phi_k\rangle,\,\sum_{j\neq k}\frac{\langle T_\Delta\phi_k,\phi_j\rangle}{\lambda_k-\lambda_j}\,\phi_j\Big).$$
This follows from first-order expansions with remainder bounds for the eigenvalue and the sign-normalized eigenfunction (e.g.~\citet[Theorems 5.1.6 and 5.1.8]{hsing2015theoretical}, for convenience we review simple sufficient bounds in Lemma~\ref{lem:eigenexpansion} in Section~\ref{sec:proofasympnorm}). Note that together with \eqref{eq:weakconvcovkernel} this implies \eqref{eq:weakinL2}.

\textit{Step 2:} The map $\chi$ from $C_s([0,1]^2)\times\R\times L_2[0,1]$ into $\R\times C[0,1]$, given by $\chi(G,\lambda,g) = \big(\lambda,\lambda^{-1}\,T_Gg\big)$ for $\lambda>0$ and defined arbitrarily for $\lambda\leq 0$, is Fr\'echet differentiable at $(\Gamma,\lambda_k,\phi_k)$, with derivative $(\Delta,s,f)\mapsto\big(s,\mc L_k(\Delta,s,f)\big)$, where
\begin{equation}\label{eq:limit}
\mc L_k(\Delta,s,f) = \lambda_k^{-1}\, T_\Delta \phi_k - s\, \lambda_k^{-2}\, T_\Gamma \phi_k + \lambda_k^{-1}\, T_\Gamma f
\end{equation}
by using the expansion \eqref{eq:basicexpansion}.

\smallskip

Then, by \eqref{eq:eigenfctexpress}, for $\|G-\Gamma\|_\infty\leq\delta_k/4$ the composite of the maps from steps 1 and 2 is the map $G\mapsto\big(\Psi_k(G),\Phi_k(G)\big)$, with the second component now as element in $C[0,1]$.

Hence by the chain rule for Fr\'echet derivatives, $G\mapsto\big(\Psi_k(G),\Phi_k(G)\big)$ is Fr\'echet differentiable at $\Gamma$ as a map from $C_s([0,1]^2)$ into $\R\times C[0,1]$. The functional delta method then actually yields joint convergence in distribution of
$\sqrt n\,\big(\hat\lambda_{k;n}-\lambda_k,\,\hat\phi_{k;n}-\phi_k\big)$ in $\R\times C[0,1]$, with an $\mathcal O_{\mathbb P}(n^{-1/2})$-remainder term, and the weak limit of  $\sqrt n\,\big(\hat\phi_{k;n}-\phi_k\big)$ in $C[0,1]$ must coincide with that in $L_2$ since weak convergence in $C[0,1]$ implies weak convergence in $L_2[0,1]$.

The remaining details are given in Section~\ref{sec:proofasympnorm}.

\end{remark}

\end{revblock}

\begin{example}[Independent, fully-observed processes]\label{ex:jainmarcus}
    In the setting of the \rev{Jain--Marcus} theorem\rev{ \citep{jain1975central}}, suppose that $(Z_n)_{n\geq 1}$ are i.i.d.~centered processes on $[0,1]$ having the same distribution as $Z$. Assume $\expec[Z(t)^4] < \infty$, $t \in [0,1]$, and that for some $\eta \in (0,1]$,
    $$ |Z_i(t) - Z_i(s)| \leq M_i\, |t-s|^\eta, \qquad \rev{s,t\in[0,1], \ \text{almost surely},}$$
    where the Hölder constants $M_i$ are i.i.d.~having finite \rev{fourth} moment, $\expec[M_i^4]<\infty$. Consider $\hat \Gamma(s,t) = \frac1n \sum_{i=1}^n Z_i(s)\, Z_i(t).$ \rev{Here $Z_i$ is jointly measurable and a random element of $C[0,1]$. Further, the bound $|Z_i(t)|\leq  |Z_i(0)| + M_i = N_i$ ensures integrability of all products $Z(x)Z(y)Z(s)Z(t)$ of up to four random variables,} and we have
    $$ |Z_i(s)\, Z_i(t) - Z_i(s')\, Z_i(t')| \leq 2\, N_i^2\, \max(|t-t'|^\eta ,|s - s'|^\eta). $$
    Therefore, the Jain-Markus theorem applies to the bivariate process $X_i(s,t) = Z_i(s)\, Z_i(t)$ \citep[Theorem 15]{kato2019lecture} yielding \eqref{eq:weakconvcovkernel} with centered Gaussian $\mc G$ having as covariance kernel the \rev{fourth}-order covariance tensor
    \begin{equation}\label{eq:forthmomentoperator}
        R(x,y,s,t) \defeq \expec[Z(x)Z(y)Z(s)Z(t)]-\Gamma(x, y)\Gamma(s, t), \quad x,y,s,t \in [0,1]\,,
    \end{equation}
    and \eqref{eq:asympcent} follows. Of course, this also follows from the asymptotic expansion in \citet[Theorem 2.1] {hall2009theory}, but the assumptions (2.7) and (2.8) (arbitrarily high moments) in their theorem seem to be more restrictive. \rev{If the processes $Z_i$ are not centered, the empirical covariance kernel $\frac1n \sum_{i=1}^n (Z_i(s)-\bar Z_n(s))\,(Z_i(t)-\bar Z_n(t))$, $\bar Z_n = \frac1n\sum_{i=1}^n Z_i$, differs from $\hat\Gamma$ computed from the centered processes $Z_i - \mu$, $\mu = \expec[Z_1]$, by $(\bar Z_n - \mu)(s)\,(\bar Z_n-\mu)(t)$. However, this term is seen to be $\Op{n^{-1}}$ in the supremum norm by the central limit theorem in $C[0,1]$. Hence \eqref{eq:weakconvcovkernel} and \eqref{eq:asympcent} remain valid.}

\end{example}

\begin{example}[Cross-covariance operator and long-run covariance operator]\label{ex:longrun}
    Theorem \ref{th:asympnorm} generalizes to estimates $\widehat G_n(x,y)$ of continuous, symmetric kernels $G(x,y)$, $x,y \in [0,1]$, these need not be positive. For example, consider a functional time series \citep{hoervathkokoska}, that is, suppose that $(Z_n)_{n \in \mathbb Z}$ is a strictly stationary sequence of centered, square-integrable processes on $[0,1]$. The lag $h$ cross-covariance kernel is $\Gamma(x,y;h) = \expec[Z_1(x)\, Z_{1+h}(y)]$, $h \in \mathbb Z$. By stationarity, $\Gamma(x,y;-h) = \Gamma(y,x;h)$, so that the kernel $\Gamma_s(x,y;h) = \Gamma(x,y;h) + \Gamma(x,y;-h)$, $h \geq 1$, is symmetric, but need not be positive. Assume \rev{fourth} moments as well as Hölder-continuous paths as in Example \ref{ex:jainmarcus}, where the Hölder constants now are a stationary sequence. The symmetric lag-$h$ cross-covariance kernel $\Gamma_{s}(x,y;h)$ is estimated by
    $$\widehat{\Gamma}_{n,s}(x,y;h) = \frac1n \sum_{j=1}^{n-h} \big(Z_j(x) Z_{j+h}(y) + Z_j(y) Z_{j+h}(x) \big). $$
    Under suitable additional conditions, e.g.~mixing as in \citet[Theorem 2.1]{dette2020functional}, a functional central limit theorem can be applied to the stationary bivariate process $Y_j(x,y) = Z_j(x) Z_{j+h}(y) + Z_j(y) Z_{j+h}(x)$, yielding
    $$\sqrt n \, \big(\widehat{\Gamma}_{n,s}(\cdot,\cdot;h) - \Gamma_s(x,y;h) \big) \stackrel{D}\rightarrow \, \mc G \qquad \text{in}\quad C([0,1]^2)$$
    with a centered Gaussian process $\mc G$, from which \eqref{eq:asympcent} follows for basis functions in the spectral decomposition of $\Gamma_s(\cdot,\cdot;h)$ with eigenvalues of multiplicity one.

    Further, under the mixing assumptions of \citet[Theorem 2.1]{dette2020functional} the long-run covariance kernel
    $$ C(x,y) = \Gamma(x,y) + \sum_{h=1}^\infty \Gamma_s(\cdot,\cdot;h)$$
    exists.
\rev{Let $C(x,y)$ be continuous and let $\widehat C_n$ be an estimator of $C$, for example a lag-window estimator. Suppose that we have that $r_n\,(\widehat C_n - C) \stackrel{D}\rightarrow \mc G_C$ in $C([0,1]^2)$ for a rate $r_n\to\infty$, which would be slower than $\sqrt n$. Then  Theorem~\ref{th:asympnorm} with $\sqrt n$ replaced by $r_n$ applies to estimates of the basis functions of the long-run covariance kernel associated to eigenvalues of multiplicity one.}\end{example}

\begin{revblock}
\begin{remark}[Bootstrap]\label{rem:bootstrap}
Since the delta method for the bootstrap \citep[Theorem 3.9.11]{van1996weak}  applies under  Fr\'echet differentiability (Remark \ref{rem:proofthasympnorm}), consistency in $C([0,1]^2)$ of a bootstrap approximation for $\sqrt n\,(\hgam_n-\Gamma)$ transfers to the bootstrapped FPCs. This justifies uniform confidence bands, provided the bootstrap for the estimator of the covariance kernel  is consistent.
\end{remark}
\end{revblock}

\section{Optimal rates for estimating eigenvalues and functional principal components}\label{sec:ratesasy}
\subsection{The model}
In this section we leverage optimal rates for covariance kernel estimation in a discrete observational model to principal component basis functions and eigenvalues.

Let the data $(Y_{i,  j}, x_{ j})$ be distributed according to the model \citep{berger2024dense, cai2011optimal}
\begin{align}
	Y_{i,  j}= \mu( x_{ j}) + Z_i( x_{ j}) + \e_{i,  j} \,, \quad  i=1,\dotsc,n\,, \ j = 1,\ldots,p  \,, \label{eq:model}
\end{align}
where $Y_{i, j}$ are real-valued response variables and the $x_j \in \rev{[0,1]}$ are known non-random design points which are assumed to be ordered as $x_1 < \ldots < x_{p}$. The processes $Z_1,\dotsc,Z_n$ are i.i.d.~copies of a mean-zero, square integrable random process $Z\rev{ = (Z(x))_{x\in[0,1]}}$ with  continuous covariance kernel $\Gamma(x,y) = \expec[Z(x)\,Z(y)]$. The errors $\e_{i,j}$ are independent with mean zero and are also independent of the $Z_i$, and the mean function $\mu$ is unknown.
 The number of design points $p \defeq p_n$ as well as the design points $ x_{j}= x_j(p,n)$ themselves depend on the number $n$ of functions which are observed.
\rev{Thus the observed processes $\mu + Z_i$ need not be centered. The estimators of $\Gamma$ that we consider below are based on the empirical covariance matrix of the observations, which does not depend on $\mu$, so that $\mu$ need not be estimated and no smoothness assumptions on $\mu$ are required \citep{berger2024optimal}. The lower bounds are established for $\mu=0$.}

A function $f\colon [0,1]^2 \to \R$ is H\"older-smooth with order $\gamma>0$
if for all indices $ \bs s=(s_1,s_2) \in \N_0^2$ with $\abs{\bs s} = s_1 + s_2 \leq \lfloor \gamma \rfloor =\max\{  k\in\N_0 \mid k<\gamma \} \defeql k$,
 the partial derivatives $D^{\bs s} f(v) = \partial_1^{s_1}\, \partial_2^{s_2} f(v)$, $v \in [0,1]^2$ exist and if the H\"older-norm given by
\begin{equation}\label{eq:hoeldernorm01}
\norm{f}_{\mc H, \gamma}\defeq \max\Big(\max_{\abs{\bs s}\leq k} \sup_{v} \abs{D^{\bs s} f(v)}, \max_{\abs{\bs s}=k}\sup_{v\neq w} \frac{|D^{\bs s} f(v)-D^{\bs s} f(w)|}{\norm{v-w}_\infty^{\gamma-k}}\Big)
\end{equation}
is finite, \rev{cf.}~\citet[Definition 1.2]{tsybakov2008introduction}\rev{, where only the H\"older condition on the derivatives of order $k$ is imposed}.
Define the H\"older class with parameters $\gamma>0$ and $L>0$ on $[0,1]^2$  by
\begin{equation}\label{def:hoelder:class}
    \mc H(\gamma, L) = \big\{f\colon [0,1]^2 \to \R \mid \norm{f}_{\mc H, \gamma} \leq L \big\}.
\end{equation}

\subsection{Lower Bounds}

Let us start with a lower bound on estimating the eigenvalues $\lambda_k$.
Consider the class of Gaussian processes
\begin{equation}\label{eq:Gaussianclass}
 \mc P_{\mathcal G} (\gamma;L) \defeq  \big\{ Z: [0,1] \to \R \ \text{centered Gaussian process} \mid \Gamma \in \mc H(\gamma,L)\big\}.
\end{equation}
\begin{theorem}[Lower bounds for estimating eigenvalues]\label{theorem:optimalityeigenval}
	Assume that in model \eqref{eq:model} the errors $\e_{i,  j}$ are i.i.d.~$\mathcal N(0, \sigma_0^2)$ - distributed with $\sigma_0^2 >0$.    Then for each $k$,
         \begin{align}\label{eq:lowerboundeigenvalue}
		  \liminf\limits_{n, p \to \infty} &\inf\limits_{\hat \lambda_{k; n,p}} \, \sup_{Z \in \mc P_{\mathcal G}(\gamma;L)} \, \expec\Big[   w\Big( \min(p^{\gamma},  n^{1/2})\, \big|\hat \lambda_{k; n,p} - \lambda_k \big| \Big)\Big] >0,
	\end{align}
	where the infimum is taken over all estimators $\hat \lambda_{k; n,p}$ of $\lambda_k$ and  $w$ is a  non-negative weight function \rev{on $[0,\infty)$ satisfying $\inf_{x \geq b} w(x) > 0$ for every $b>0$, in particular $w(x) = \min(x,1)$}.
\end{theorem}
\rev{Here $\liminf_{n,p\to\infty} a_{n,p} = \lim_{N\to\infty}\, \inf_{n,p\geq N} a_{n,p}$, so that the bound holds for every sequence $p=p_n\to\infty$.} The proof is provided in Section \ref{ssec:proof:optimality}.
\begin{remark}
        The lower bound of order $n^{-1/2}$ can also be obtained  for the continuous observational model $Z_1, \ldots, Z_n$, and without additional errors.

        The discretization error $p^{-\gamma}$ is well known from mean function and covariance kernel estimation \citep{berger2024dense, cai2011optimal, cai2010nonparametric, berger2024optimal} in the model \eqref{eq:model} with synchronous design points. While for discretely observed functions this is quite plausible, it is however remarkable that this term also arises for the constant eigenvalues. \rev{In a closely related setting, \citet{bourarach2026minimax} obtain a lower bound of order $n^{-1} + p^{-4\gamma}$ for the relative squared error of eigenvalue estimators, and leave the minimax rate open. Theorem~\ref{theorem:optimalityeigenval}, together with the upper bound in Theorem~\ref{thm:rates_principlecomp_estimation} (i), shows that  the minimax rate is $\max(p^{-\gamma}, n^{-1/2})$, and that it is attained uniformly over all eigenvalues. }

\end{remark}

To derive lower bounds for estimating eigenfunctions, for $k \in \N$ and $\delta >0$ we set
\begin{align}
	  \mc P_{\mathcal G} (\gamma;k, L)  & \defeq \big\{ Z \in \mc P_{\mathcal G} (\gamma;L) \mid \delta_k \geq \delta,\ \text{ where } \delta_k \text{ is the } k^{\text{th}} \label{eq:fctclasswithspecgap}\\
    & \qquad \qquad \text{spectral gap of the covariance kernel } \Gamma_Z\big\},\nonumber
\end{align}
see \eqref{eq:specgapcovkern} for the definition of $\delta_k$. \rev{Note that since $\delta_k \leq \lambda_k$ we must have $\lambda_j \geq 2\,\delta_k$ for $j<k$, and thus we get for $\Gamma \in \mc H(\gamma,L)$ that
$$(2k-1)\, \delta_k \leq \sum_{j \leq k}\lambda_j \leq \sum_{j \geq 1}\lambda_j = \int_0^1 \Gamma(x,x)\,\dx x \leq L.$$ Hence we also have $\delta \leq L/(2k-1)$ for the function class in \eqref{eq:fctclasswithspecgap} to be non-empty.}

Moreover, observe that stating lower bounds in the supremum norm $\normb{\hat \phi_{k; n,p}  -  \phi_k}_\infty$ is not meaningful since the appropriate sign of $\phi_k$ cannot be determined by the data. Indeed, one could devise lower bounds which show that no consistent estimation in the sup-norm is possible without adapting the sign based on the unknown function $\phi_k$.
Therefore, to derive lower bounds we introduce the following semimetric on $C[0,1]$,
\begin{equation}\label{eq:semimetric}
\dd(\phi, \psi)_{\infty, s} = \min\big( \normb{\phi  -  \psi}_\infty, \normb{\phi  +  \psi}_\infty\big), \qquad \phi, \psi \in  C[0,1].
\end{equation}
Note that in \eqref{eq:supboundeigenfct}, with $\hat \phi_k = \hat \phi_{k,n}$ based on $n$ data points, consistency $\big\| \hat \phi_{k,n} - c_k\, \phi_k\big\|_\infty \to 0$ in probability as $n \to \infty$ implies that
\begin{equation}\label{eq:coincidebound}
\prob\Big(\big\| \hat \phi_{k,n} - c_k\, \phi_k\big\|_\infty = \dd(\hat \phi_{k,n}, \phi_k)_{\infty, s}\Big) \to 1,\quad n \to \infty,
\end{equation}
so that upper bounds derived from \eqref{eq:supboundeigenfct} can also be stated in the semimetric \eqref{eq:semimetric}.

\begin{theorem}\label{theorem:optimality}
	Assume that in model \eqref{eq:model} the errors $\e_{i,  j}$ are i.i.d.~$\mathcal N(0, \sigma_0^2)$ - distributed, $\sigma_0^2 >0$\rev{, and that $\delta \leq c\,L$ for a sufficiently small constant $c>0$ depending only on $\gamma$ and $k$}. Then setting
    \begin{equation}\label{eq:ratethe}
    a_{n,p} = \max\Big(p^{-\gamma}, \Big(\frac{\log(n\,p)}{n\,p}\Big)^{\frac \gamma{2\gamma + 1}},  n^{-1/2} \Big)
    \end{equation}
	we have that
	\begin{align}\label{eq:lowerboundeigenfunction}
		\liminf\limits_{n, p \to \infty} &\inf\limits_{\hat \phi_{k; n,p}} \, \sup_{Z \in \mc P_{\mathcal G}(\gamma; k, L)} \, \expec\, \big[ w\big(a_{n,p}^{-1} \cdot \dd(\hat \phi_{k; n,p},  \phi_k)_{\infty, s}   \big)\big] >0\,.
	\end{align}

	where the infimum is taken over all estimators $\hat \phi_{k; n,p}$ of $\phi_k$ and $w(x)$ is as in Theorem \ref{theorem:optimalityeigenval}.
\end{theorem}
\begin{remark}[Relation of \eqref{eq:lowerboundeigenfunction} to \citet{belhakem2025minimax, bourarach2026minimax}]
    Let us comment on the rates in \eqref{eq:lowerboundeigenfunction} and compare them to the results in \citet{belhakem2025minimax, bourarach2026minimax}, who obtain lower bounds in $L_2$. The lower bound of order $p^{-\gamma}$  is also obtained in \citet{belhakem2025minimax} for $\alpha \leq 1$, and in \citet{ bourarach2026minimax} for general Hölder regularity $\alpha >0$. Our construction seems to be slightly simpler since only two \rev{hypotheses need to be compared, that is, two distributions of the process $Z$ in $\mc P_{\mathcal G}(\gamma;k,L)$ which induce the same distribution of the observations at the design points, but whose $k$-th FPCs are at distance of order $p^{-\gamma}$.}

    The lower bound of order $(\log(n\,p)/(n\,p))^{\gamma/(2\gamma + 1)}$ is particular for the supremum norm, and does not occur for the $L_2$-norm in \citet{belhakem2025minimax, bourarach2026minimax}.
    Our construction, provided in Section \ref{ssec:proof:optimality}, is based on that in \citet{berger2024optimal} for the covariance kernel, but modifications are required\rev{, since the perturbed functions must have unit $L_2$-norm so that they are the FPCs of the rank-one covariance kernels of the hypotheses}.

    \rev{For these two terms it suffices to consider processes with rank one covariance kernels. For the lower bound of order $n^{-1/2}$, however, processes with covariance kernels of rank at least two are required as in \citet{belhakem2025minimax} and \citet{bourarach2026minimax}, since for continuously observed paths without noise, a single path of a rank-one process already determines its FPC up to sign; see also \citet{mas2015high} for a lower bound for the empirical eigenprojections of fully observed processes.}

\end{remark}

\vspace{2mm}

\subsection{Upper bounds }

Here we show that the lower bounds in Theorems \ref{theorem:optimalityeigenval} and \ref{theorem:optimality} are indeed optimal, by plugging the covariance kernel estimator from \citet{berger2024optimal} with its rates into \eqref{eq:spectraleigenvalues} and \eqref{eq:supboundeigenfct}. Let us briefly recall the estimator from \citet{berger2024optimal}, see also \citet{mohammadi2024functional}, together with the more general class of processes compared to the Gaussian class $
 \mc P_{\mathcal G} (\gamma;L) $ in \eqref{eq:Gaussianclass}, over which it achieves optimal rates.
A symmetric function such as $\Gamma$ is uniquely determined by its values on the upper triangle
\begin{align}
	T \defeq \{(x,y) \in  [0,1]^2 \mid x \leq y\}. \label{eq:upper_triangle}
\end{align}
For a sufficiently smooth function $f\colon T \to \R$ we define the Hölder norm $\norm{f}_{\mc H, \gamma;T}$ of $f$ on $T$ of order $\gamma>0$ analogously to  $\norm{\cdot}_{\mc H, \gamma}$ in \eqref{eq:hoeldernorm01}, with the suprema in the definition being only taken over $T$.
Further, define the H\"older class $\mc H_{T}(\gamma, L)$ with parameters $\gamma>0$ and $L>0$ on $T$  analogously to \eqref{def:hoelder:class}.
Note that for a symmetric function $f$ on $[0,1]^2$,  $f_{|T} \in \mc H_{T}(\gamma, L)$ for $\gamma \leq 1$ implies $f \in \mc H_{[0,1]^2}(\gamma, 2\,L)$, but for $\gamma > 1$, $f_{|T} \in \mc H_{T}(\gamma, L)$ only implies $f \in \mc H_{[0,1]^2}(1, 2\,L)$. Thus, a symmetric function $f \in \mc H_{[0,1]^2}(1, L)$, when restricted to $T$\rev{,} can be contained in a higher-order Hölder class \rev{$\mc H_{T}(\gamma, \tilde L)$, $\gamma>1$; an example is the covariance kernel $\min(x,y)$ of Brownian motion, which is linear on $T$ but not differentiable on the diagonal}.

For the process $Z$ in \eqref{eq:model}  we relax the Gaussianity assumption and assume that $\expec[Z(0)^4] < \infty$. Further, assume that the paths are Hölder continuous of some potentially low order: there exists $0 < \beta \leq 1$ and a random variable $M = M_Z>0$ with $\expec [M^4] < \infty $ such that
\begin{equation}\label{eq:hoeldercontpathsZ}
	\big|Z( x)-Z( y)\big| \leq M \, | x- y|^\beta, \qquad x,  y \in [0,1] \quad \text{almost surely.}
\end{equation}

\rev{In particular, $Z$ has a.s.~continuous paths,
and $|Z(x)| \leq |Z(0)| + M$ a.s., whence $\sup_{x}\expec[Z(x)^4] \leq 8\,(\expec[Z(0)^4] + \expec[M^4])$.} Then, given $C_Z>0$ and $0 < \beta_0 \leq 1$, we consider the class of processes
\begin{align}
	\Zclass \defeq \mc P (\gamma;L,\beta_0, C_Z) & = \big\{ Z: [0,1] \to \R \ \text{centered random process} \mid \exists \ \beta \in [\beta_0,1] \text{ and } M \text{ s.th.}\nonumber\\
	& \qquad \quad \expec [M^4] + \expec [Z( 0)^4]\leq C_Z \text{, \eqref{eq:hoeldercontpathsZ} holds} \label{eq:classprocesses} \text{ and  } \Gamma_{\mid T} \in \mc H_T(\gamma,L)\big\},
\end{align}
where $\Gamma$ is the covariance function of $Z$. As discussed in \citet{berger2024optimal}, Theorem 1 in \citet{azmoodeh2014necessary} implies that $\rev{\mc P_{\mathcal G} (\gamma;L) \subseteq \mc P (\gamma;L,\beta_0,C_Z)}$ for appropriate constants $\beta_0$ and $C_Z$.

\medskip

\citet{berger2024optimal} consider linear estimators for $\Gamma$ at $(x,y) \in T$ of the form

\begin{equation}
\color{black}
    \covest xyh \defeq  \sum_{j<k}^p \wjk xyh \, z_{j,k;n} , \quad \text{where } \quad    z_{j,k;n} \defeq  \frac1{n-1}\sum_{i=1}^n\big(Y_{i,j} Y_{i,k} - \bar Y_{n, j} \bar Y_{n, k}\big)\label{eqn:estimatorCovariance}
\end{equation}
with   $\bar Y_{n,j } = n^{-1}\sum_{i=1}^n Y_{i,j}$,  $j,k \in \{1, \ldots, p\}$. The $z_{j,k;n}$ are the entries of the empirical covariance matrix; $h>0$ is a bandwidth parameter and  $\wjk xyh = w_{j,k;p}(x,y;h;x_{1},\dotsc,x_{p})$ are weights.
 For $x^\prime, y^\prime \in [0,1]$ with $x^\prime>y^\prime$ we simply set
$\hat \Gamma_n(x^\prime,y^\prime;h) : = \hat \Gamma_n(y^\prime,x^\prime;h)$.  A special case is a restricted local polynomial estimator as discussed in \citet{berger2024optimal}.

We use the basis functions and the eigenvalues of $\covest xyh $, denoted by $\phikest xh$ and $\lamkest h$, as estimates of $\phi_k$ and $\lambda_k$.
Now, for $k \in \N$ and $\delta >0$ we set
\begin{align}
	\Zclassk \defeq \mc P (\gamma;k,L,\beta_0, C_Z, \delta) & = \big\{ Z \in \mc P (\gamma;L,\beta_0, C_Z) \mid \delta_k \geq \delta,\ \text{ where } \delta_k \text{ is the } k^{\text{th}} \label{eq:fctclasswithspecgap1}\\
    & \qquad \qquad \text{spectral gap of the covariance kernel } \Gamma_Z\big\},\nonumber
\end{align}
analogously to \eqref{eq:fctclasswithspecgap}, and observe that $\mc P_{\mathcal G} (\gamma;k, L) \subseteq \Zclassk $.

\rev{To make precise which results from \citet{berger2024optimal} are used, Assumption~\ref{ass:weights} on the weights of the linear estimator in \eqref{eqn:estimatorCovariance} and the resulting rates for the covariance kernel estimator in Lemma~\ref{thm:rates_covkernrates} are stated in Appendix~\ref{sec:appendix1}.} Consider the following assumptions on the distribution of the errors and on the design, for further discussion of these assumptions see \citet{berger2024optimal}.

\begin{assumption}[Sub-Gaussian errors] \label{ass:distribution}
		The random variables $\{\epsilon_{i, j} \mid 1 \leq i \leq n,\,  1 \le  j \le  p\}$ are centered, independent and independent of the processes $Z_1,\dotsc,Z_n$. Further we assume
        \begin{enumerate}
        \item[(i)] $ \sigma_{ij}^2 \defeq \expec[\epsilon_{i, j}^2]$ we have that $\sigma^2 \defeq \sup_n \max_{ i,j} \sigma_{ij}^2 < \infty$,
        \item[(ii)] in addition to (i), the distribution of $\epsilon_{i, j}$ is sub-Gaussian, and there exists {\color{black}$\kappa\geq 1$} such that {\color{black}$\kappa^2\sigma_{i, j}^2$} is an upper bound for the squared sub-Gaussian norm of $\epsilon_{i, j}$.
        \end{enumerate}
\end{assumption}

\begin{assumption}[Design Assumption] \label{ass:design:localization}
	There is a constant $\Ccard > 0$ such that for each $x \in \rev{[0,1]}$ and $h>0$ we have that
	\begin{align*}
		\card \big\{ j \in \{1, \ldots, p \}\mid x_j \in [x-h, x+h]\big\} & \leq \Ccard \,\rev{\max(p\,h, 1)}\,.
	\end{align*}
\end{assumption}

\rev{The proof of the following theorem is given in Section~\ref{sec:upperboundsproofs}.}

\begin{theorem}\label{thm:rates_principlecomp_estimation}
	Consider model \eqref{eq:model} under the design Assumption~\ref{ass:design:localization}. Suppose that for given $\gamma>0$ the weights in the linear estimator $\hat \Gamma_{n}(\cdot; h)$ for the covariance kernel $\Gamma$ in \eqref{eqn:estimatorCovariance}  satisfy Assumption \ref{ass:weights}, (W1) with $\zeta = \floor \gamma$, (W2) and (W3). Suppose that $0 < \beta_0 \leq 1$ and $L, C_Z>0$.
    \begin{enumerate}
     \item Then under Assumption \ref{ass:distribution}, (i) on the errors,
	 setting
	$	\bar h^\star \sim \max \Big( c/p, \Big(\frac{1}{n\,p}\Big)^{\frac1{2\gamma + 1}} \Big)\,$
	we have that
     \begin{align}\label{eq:rateofconvcovkernel1}
		 \sup_{Z \in \mc P (\gamma)} \,  \expec\big[  \sup_{k \geq 1} \big( \lamkest {\bar h^\star} - \lambda_k\big)^2\big] = \mc O\Big(\rev{\max\Big(p^{-2\gamma},  n^{-1} \Big)}\,\Big)\,.
	\end{align}
       \item If Assumption \ref{ass:distribution}, (i) and (ii) and additionally Assumption \ref{ass:weights}, (W4) holds, then
 for $k \in \N$, $\delta>0$,
     setting
	$	h^\star \sim \max \Big( c/p, \Big(\frac{\log(n\,p)}{n\,p}\Big)^{\frac1{2\gamma + 1}} \Big)\,$
	we have that
     \begin{align}\label{eq:rateofconvcovkernel}
		 \normb{\phikest \cdot{h^\star} - \rev{c_k}\, \phi_k}_\infty = \mc O_{\prob;  \, Z \in \Zclassk}\big(a_{n,p} \big)\,.
	\end{align}
    where \rev{$c_k = \sign(\langle \phikest\cdot{h^\star}, \phi_k\rangle)$ as in \eqref{eq:signcoefficient},} the rate $a_{n,p}$ is as in \eqref{eq:ratethe}, and the uniform rate for convergence in probability is defined as
    $$\lim_{C \to \infty}\, \limsup_{n,p \to \infty} \, \sup_{Z \in \Zclassk}\, \prob_{Z}\big( a_{n,p}^{-1}\, \normb{\phikest \cdot{h^\star} - \rev{c_k}\, \phi_k}_\infty \geq C\big) = 0.$$
    \end{enumerate}
\end{theorem}

\begin{remark}[On the proof]
In Theorem \ref{thm:rates_principlecomp_estimation}, (ii) follows from \eqref{eq:supboundeigenfct} and the sup-norm rates for the covariance kernel estimation in \citet{berger2024optimal}, while  (i) is obtained from \eqref{eq:spectraleigenvalues}, the bound $\big\|\hat T - T_\Gamma \big \|_{L_2} \leq \big\|\hat \Gamma - \Gamma \big \|_{L_2([0,1]^2)}$ and rates in $L_2$ for the estimator $\covest xyh$. See Section \ref{sec:upperboundsproofs} for the details.
\end{remark}

\begin{remark}[Choice of the bandwidth]
    While the bandwidths $h^\star$ and $\bar h^\star$ can be of different order, in the dense regime in which the $\sqrt n$ can be achieved the specific choice is not required. Indeed, \rev{similarly} to \citet[Remark 5]{berger2024dense}, if $ (\log n)^{\rev{\nu + \vartheta}} \, n^{1/(2 \gamma)} \lesssim p$ for a $\rev{\vartheta} >1$ and $\rev{\nu} >0$, then choosing $h_n \eqsim (\log n)^{\rev{\vartheta}} /p$ one obtains the $\sqrt n$-rate in both \eqref{eq:rateofconvcovkernel1} and \eqref{eq:rateofconvcovkernel}.
\end{remark}

\begin{remark}[Asymptotic normality]
    \citet[Theorem 3]{berger2024optimal} also provide a CLT in $C([0,1]^2)$ for the covariance kernel estimator $\rev{\hgam_n(\cdot\,;h_n)}$,  along appropriate sequences of bandwidths $(h_n)$ in the $\sqrt n$-regime, where the limit Gaussian process has covariance function \eqref{eq:forthmomentoperator}. Therefore, Theorem \ref{th:asympnorm} applies to yield asymptotic normality in $C[0,1]$ of $\phikest \cdot {h_n}$.
\end{remark}

\begin{remark}[FPCs of the derivative process]
Suppose that the centered, square-integrable process $(Z(x))_{x \in [0,1]}$ with covariance kernel $\Gamma$ has differentiable sample paths. \citet{dai2018} discuss the derivative process $Z'$, assumed to be square-integrable as well, and its \rev{principal} component decomposition. The process $Z'$ has covariance kernel $\Gamma^{(1,1)}$, where $\Gamma^{(k,l)}(x,y) = \partial_x^k\, \partial_y^l\, \Gamma(x,y)$, $k,l \in \N_0$.
Then, as highlighted in \citet{dai2018}, the FPCs $\phi^{\text{d}}_n$, $n \in \N_0$ associated to $Z'$ and $\Gamma^{(1,1)}$ are not equal to the derivatives $\phi_n'$ of the FPCs $\phi_n$ of $\Gamma$, and neither need the eigenvalues coincide.

\citet{dai2018} provide rates of convergence in sup-norm of pooled local quadratic estimates \citep{zhang2016sparse}  of  $\Gamma^{(1,1)}$ and resulting estimates of FPCs $\phi^{\text{d}}_n$ in a discrete observational model with asynchronous design.
\citet{berger2024optimal} analyze an estimator in a deterministic, synchronous design which is analogous to \eqref{eqn:estimatorCovariance} with restricted local polynomial weights $w_{j,k}^{(1,1)}(x,y;h)$ for derivative estimation. Under sufficient regularity of the sample paths and the kernel $\Gamma$, for estimating $\Gamma^{(1,1)}$ they obtain the rate
    \begin{equation}\label{eq:ratederiv}
    a^{\text {d}}_{n,p} = \max\Big(p^{-\gamma+2}, \Big(\frac{\log(n\,p)}{n\,p}\Big)^{\frac {\gamma-2}{2\gamma + 1}},  n^{-1/2} \Big),
    \end{equation}
which can be transferred to the FPCs.
\end{remark}

\section{Numerical methods, simulations and real data illustration}\label{sec:numerics}

\rev{The code for all numerical results of this section, including the real-data application, is} available in the Github repository\footnote{github.com/KevinWilk/Functional-principal-component-estimation \href{https://github.com/KevinWilk/Functional-principal-component-estimation}{[github]}}.

\subsection{Numerical approximation of eigenfunctions and eigenvalues} \label{sec:numericapprox}

To compute eigenfunctions and eigenvalues of an estimate $G=\hgam_n$ with suitable accuracy we rely on the Nyström method \citep[Chapter 12]{kress2014integralequi}, which in turn is based on numerical integration by quadrature.
Take $m$ quadrature points $\tilde{x}_1 < \ldots < \tilde{x}_m $ with \rev{positive} quadrature weights $w_1\rev{,} \ldots\rev{,} w_m $.
Instead of solving $T_G \varphi = \lambda \, \varphi$, compute solutions $\varphi_m, \lambda_m$ to $T_m \varphi_m = \lambda_m \, \varphi_m$ where
$$ T_m \varphi(x) = \sum_{l=1}^m w_l\, G(x,\tilde{x}_l)\, \varphi(\tilde{x}_l).$$
To this end, form the symmetric $m \times m$-matrix $S$ with
$$S_{i,l} = \sqrt{w_i}\,G(\tilde{x}_i,\tilde{x}_l)\, \sqrt{w_l}, $$
and compute an eigenvalue-eigenvector pair $\lambda_m$ and $v_m \in \R^m$ for $S$. Note that $u_m$ with $u_{m,j} = v_{m,j}/\sqrt{w_j} $ is an eigenvector for the non-symmetric Nyström matrix $K$ with $K_{i,l} = \rev{G(\tilde{x}_i,\tilde{x}_l)\,w_l}$, however, it may be preferable numerically to work with the symmetric version $S$. Then an eigenvalue-eigenfunction pair for $T_m$ \rev{is} given by $\lambda_m>0$ and $\varphi_m$ with
\begin{equation}\label{eq:nystroem}
 \varphi_m(x) = \frac{1}{\lambda_m}\sum_{l=1}^m \sqrt{w_l}\, G(x,\tilde{x}_l)\, v_{m,l},
 \end{equation}
\rev{where $v_m$ is a unit vector. Then $\varphi_m(\tilde x_j) = u_{m,j}$, $j=1,\ldots,m$, so that $\varphi_m$ is normalized with respect to the quadrature rule, $\sum_{l=1}^m w_l\, \varphi_m(\tilde x_l)^2 = 1$.}
In our experiments we restrict ourselves to the Nyström method with composite trapezoidal and composite Simpson rule, see \citet[Theorems 12.1 and 12.2]{kress2014integralequi}.

Convergence analysis as the number of quadrature points $m \to \infty$ relies on the fact that for the Nyström method   the resulting family of operators is pointwise convergent to $T_G$ and collectively compact if the quadrature family is convergent \citep{zbMATH03240726, kress2014integralequi}. Then one may refer to perturbation theory for collectively compact and strongly convergent families, which is developed in \citet{zbMATH03253481}.

\subsection{Simulation results}\label{sec:sim}

As data generating process we take a stationary Ornstein–Uhlenbeck (OU) \rev{process}
\begin{align*}
    Z(x) = Z(0) e^{-\theta x}  + \sigma \int_0^x e^{-\theta(x-y)} \rd B_y,
\end{align*}
where $Z(0)\sim \mathcal{N}(0,\sigma^2/(2\theta))$  and $(B_t)_{t \geq 0}$ is a standard Brownian motion independent of $Z(0)$, and we take $\sigma=2$ and $\theta=3$.
The covariance kernel is given by
\begin{align}
    \Gamma(x,y) =   \frac{\sigma^2}{2 \theta}  e^{-\theta|x-y|}. \label{eq:cov kernel}
\end{align}
and the FPCs and associated eigenvalues are given by \citep{Ghanem1991}
\begin{align}
    \phi_k(x) = A_k \big( \cos(\omega_k x) + \frac{\theta}{\omega_k} \sin(\omega_k x) \big) \quad \text{ and } \quad \lambda_k = \frac{\sigma^2}{\theta^2 + \omega_k^2}, \label{eq:analytic PCBF}
\end{align}
where
$$A_k = \Big( \int_0^1 \big(\cos(\omega_k y) + \frac{\theta}{\omega_k} \sin(\omega_k y)\big)^2 \rd y \Big)^{-1/2}$$
and $\omega_k>0$ is the solution of
$$(\theta^2-\omega_k^2)\sin(\omega_k) + 2 \theta \omega_k \cos(\omega_k) = 0.$$
Moreover, we use model \eqref{eq:model} with \rev{$\mu = 0$ and the} equidistant design $x_j = j/p$, $j = 0,\hdots,p$\rev{, which consists of $p+1$ design points, and with independent} standard normally distributed errors $\epsilon_{ij}$. \rev{Thus the error variance is of the same order as the variance $\sigma^2/(2\theta) = 2/3$ of $Z(x)$. Recall that the} mean function cancels out when forming the empirical covariance matrix in \eqref{eqn:estimatorCovariance}.

The covariance kernel is estimated using the estimate \eqref{eqn:estimatorCovariance} with the restricted local linear weights from \citet{berger2024optimal} as implemented in the \texttt{R}-package \texttt{biLocPol}\footnote{contained in Github repository github.com/mbrgr/Optimal-Rates-Covariance-Kernel-Estimation-in-FDA \href{https://github.com/mbrgr/Optimal-Rates-Covariance-Kernel-Estimation-in-FDA}{[Link]}}. Estimates for the sample size $n=500$ and $p=50$ for a particular sample, with bandwidths $0.65$ and $0.25$ together with the population kernel \eqref{eq:cov kernel} are  displayed in Figure \ref{fig: est cov} in \rev{Appendix \ref{sec:trapez}}.

For computing the associated FPCs in the Nyström method \eqref{eq:nystroem} we use $201$ equidistant quadrature points and the composite Simpson rule. \rev{We write $\hat\phi_{k;n,m}(\cdot\,;h)$ and $\hat\lambda_{k;n,m}(h)$ for the resulting approximations, based on $m$ quadrature points, of the $k$-th FPC and the $k$-th eigenvalue of $\hgam_n(\cdot,\cdot\,;h)$, and $\hat c_{k;m} = \sign\big(\langle \hat\phi_{k;n,m}(\cdot\,;h),\phi_k\rangle\big)$ for the sign in \eqref{eq:signcoefficient}. Supremum norms are evaluated on the grid of quadrature points.} Figure \ref{fig:PCB functions} shows plots for the first two estimated FPCs together with their population counterparts, and provides the estimates and true values for the eigenvalues. Specifically we use $\hat{\phi}_{1;500,201}(\cdotp;0{.}65)$ respectively $\hat{\phi}_{2;500,201}(\cdotp;0{.}25)$, where $0.65$ and $0.25$ are the bandwidths used in the covariance kernel estimate.

\rev{The bandwidths $0.65$ and $0.25$ are oracle choices. In a simulation with $N=1000$ repetitions we determine, for different numbers $p$ and for each bandwidth $h$ in a grid with step size $0.05$, the mean of the sup-norm error $\|\hat{\phi}_{k;500,201}(\cdotp;h) - \hat{c}_{k;201} \, \phi_k \|_{\infty}$, see Figure \ref{fig: optimal slection} (left panel). For each $p$ the minimizer $h^* = h^*(n,p;k)$ is marked at the bottom of the panel, and $0.65$ and $0.25$ are the minimizers for $\phi_1$ and $\phi_2$ for $p=50$. The oracle bandwidth depends on $k$ since $\phi_2$ has more variation than $\phi_1$: for $\phi_2$ the smoothing bias leads to a marked increase of the error for $h \geq 0.4$, whereas for $\phi_1$ the mean error varies little for $h$ between $0.4$ and $1$. Thus the minimizer for $\phi_1$ is not well determined, and the improvement over smaller bandwidths such as $0.25$ is only moderate. The right panel in Figure \ref{fig: optimal slection} has the mean sup-norm error of the covariance kernel estimate itself, which is minimal for bandwidths between $0.2$ and $0.4$. If this bandwidth is used for both FPCs instead of $h^*(n,p;k)$, then for $n=500$ and each $p$ in Figure \ref{fig: supnorm error} the mean sup-norm error increases by $9\%$ to $17\%$ for $\phi_1$ (for $p=50$ from $0.102$ to $0.111$) and by at most $2\%$ for $\phi_2$. Hence a single bandwidth which is suitable for the covariance kernel is also reasonable for both FPCs. In practice, one may therefore choose the bandwidth for covariance kernel estimation, e.g.~using the cross validation method from \citet{berger2024optimal}.}

\begin{figure}[h!]
    \centering
    \begin{subfigure}{0.49\linewidth}
    \includegraphics[width=\linewidth]{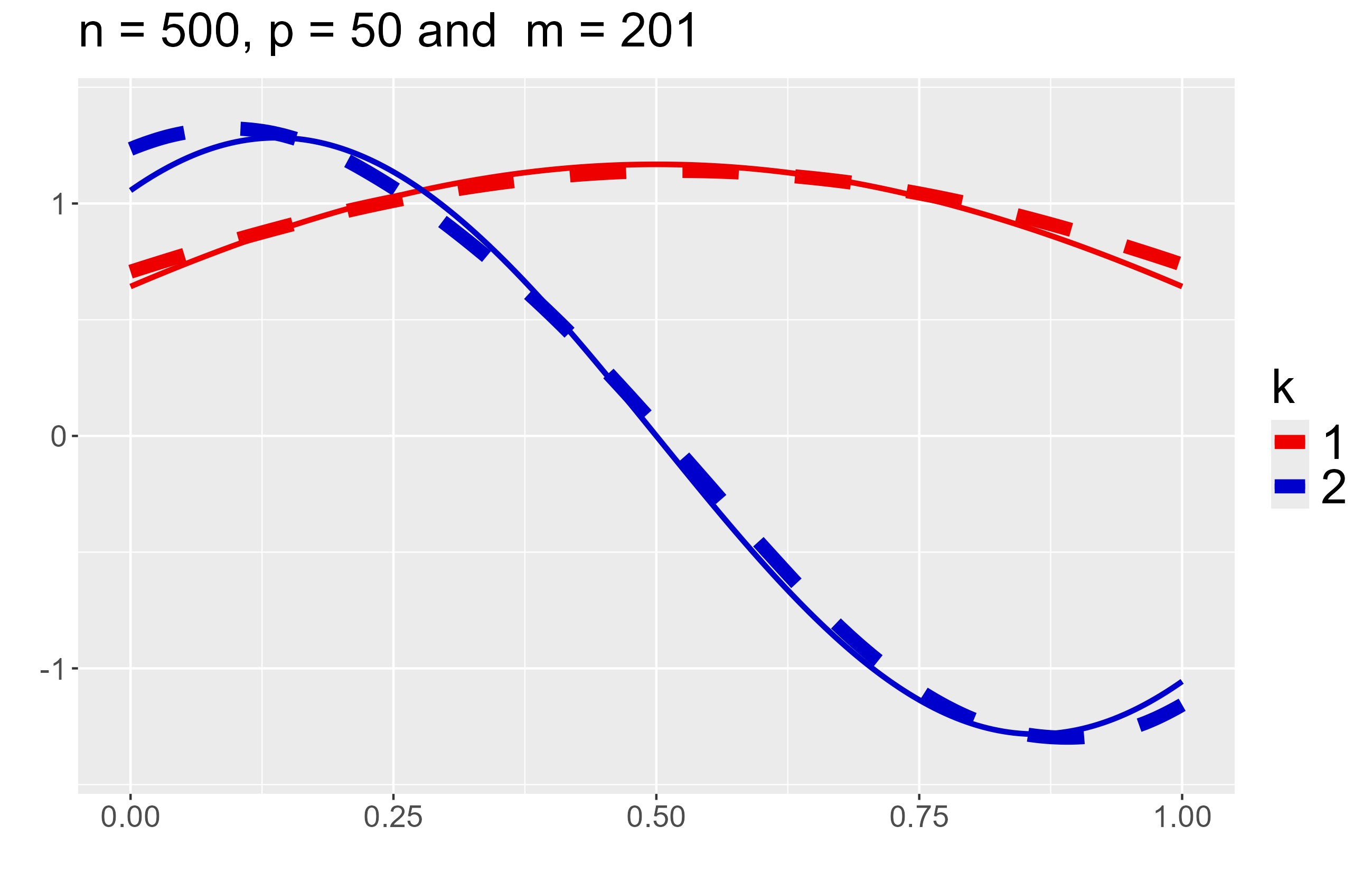}
    \end{subfigure}
    \hfill
    \raisebox{2cm}{
    \small
    \centering
    \begin{tabular}{c|c|cc}
                                     & $k = 1$       & $ k = 2$     &  \phantom{xxxxx}   \\  \hline
     $\lambda_{k}$                   & $0{.}3099$    & $0{.}1433$   &     \\
     $\hat{\lambda}_{k;500,201}(h)$    & $0{.}3190$    & $0{.}1443$ &     \\
     $h$                             & $0{.}65$      & $0{.}25$     &
    \end{tabular}
    }
    \caption{Estimation of principal component basis functions and corresponding eigenvalues: analytical $\phi_k$ (solid line) and $\hat{\phi}_{1;500,201}(\cdotp;0{.}65)$ and $\hat{\phi}_{2;500,201}(\cdotp;0{.}25)$ (dashed line).}
    \label{fig:PCB functions}
\end{figure}

\begin{figure}[h!]
    \centering
    \begin{subfigure}{0.49\linewidth}
    \includegraphics[width=\linewidth]{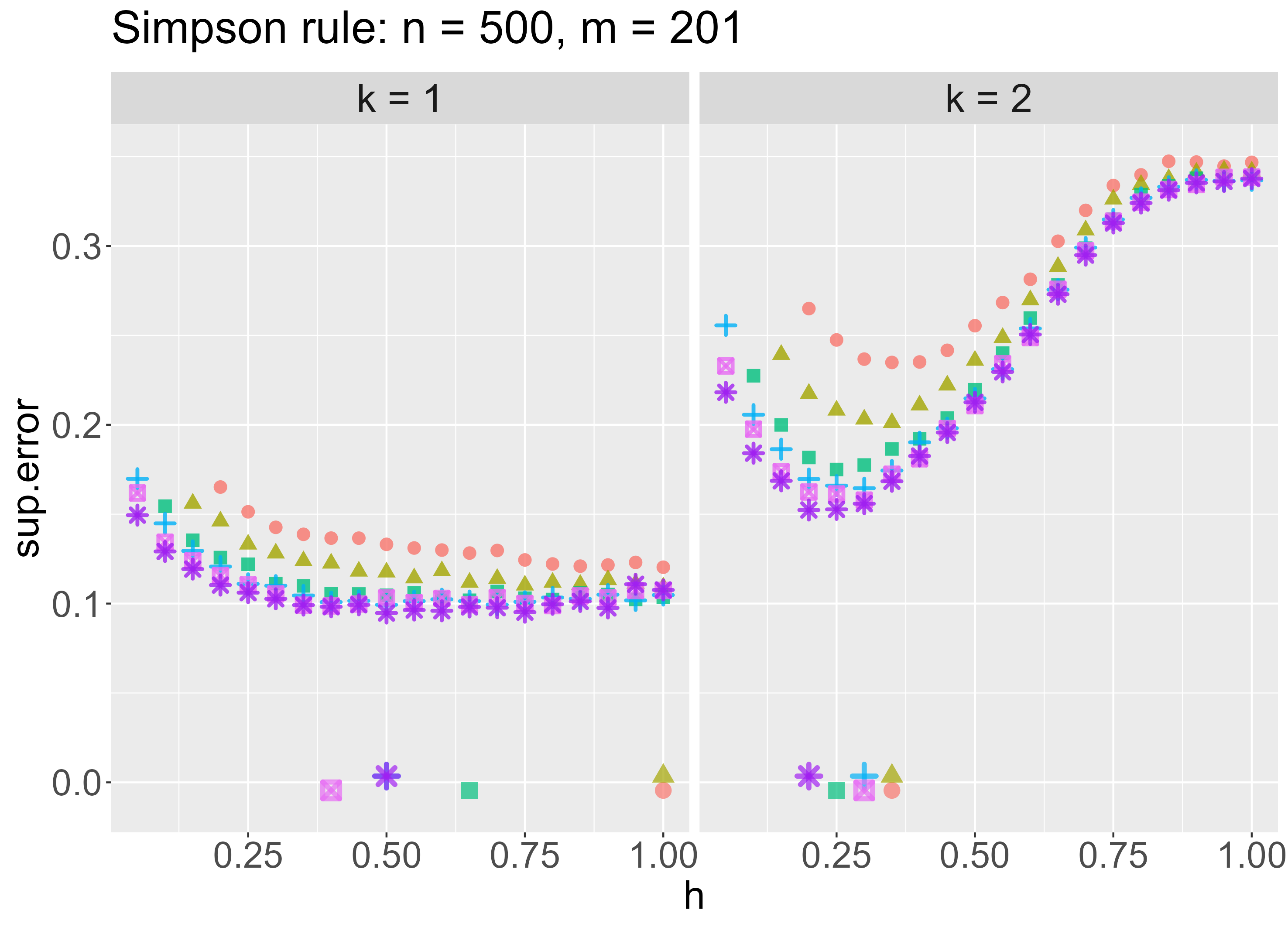}
    \caption{\small $\|\hat{\phi}_{k;500,201}(\cdotp;h) - \hat{c}_{k;201} \, \phi_k \|_{\infty}$.}
    \end{subfigure}
    \hfill
    \begin{subfigure}{0.49\linewidth}
    \includegraphics[width=\linewidth]{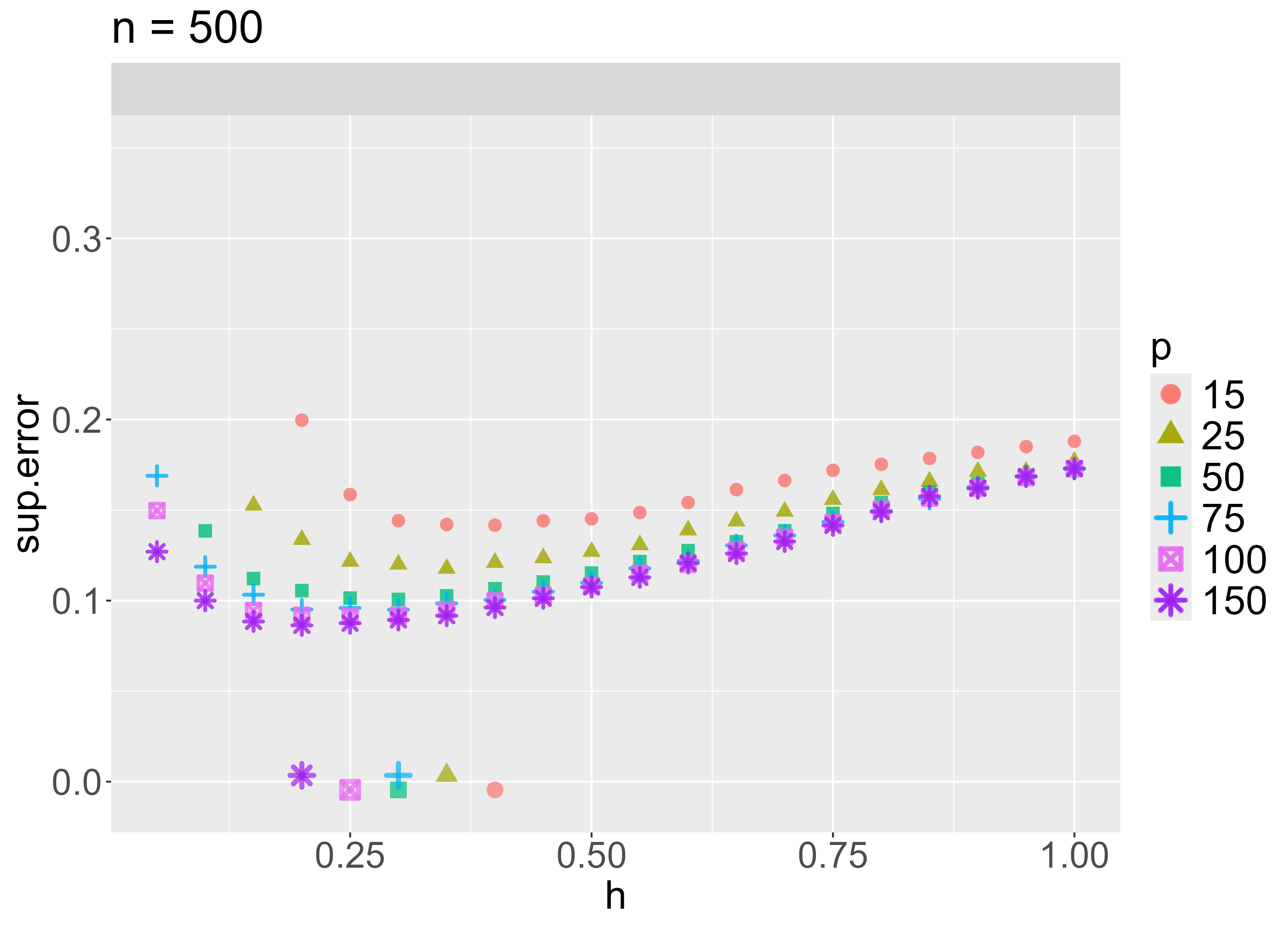}
    \caption{\small $\|\hgam_{500}(\cdotp,\cdotp;h) - \Gamma \|_{\infty}$.}
    \end{subfigure}
    \caption{Mean of supremum norm error of $N = 1{,}000$ iterations with estimator $\hgam_{500}$
of degree one for estimating the principal component basis functions \eqref{eq:analytic PCBF} and covariance kernel \eqref{eq:cov kernel} of the OU-process.}
    \label{fig: optimal slection}
\end{figure}
Finally, in Figure \ref{fig: supnorm error} we investigate the impact of $p$ for different sample sizes: For each $n \in \{50,100,200,350,500\}$ and $p \in \{15,25,50,75,100,125,150\}$ we determined the \rev{oracle bandwidth} $h^* = h^*(n,p;k)$ in terms of \rev{mean} sup-norm error, tuned for each basis function $\phi_k$, $k=1,2$, and plot the resulting  mean sup-norm error $\|\hat{\phi}_{k;n,201}(\cdotp;h^*) - \hat{c}_{k;201} \, \phi_k \|_{\infty}$, and also the \rev{mean relative} error of the corresponding eigenvalue estimate $|(\hat{\lambda}_{k;n,201}(h^*) - \lambda_{k})/\lambda_{k}|$ \rev{at the same bandwidth}. For each $n$, when increasing $p$ from $15$ to $25$ and up to $75$ the average errors  for estimating both $\phi_k$ and $\lambda_k$ first decrease at least moderately, and then stabilize for larger values of $p$. Similar results when using the composite trapezoidal rule are provided in \rev{Appendix \ref{sec:trapez}}.

\begin{figure}[h!]
    \centering
    \begin{subfigure}{0.49\linewidth}
    \includegraphics[width=\linewidth]{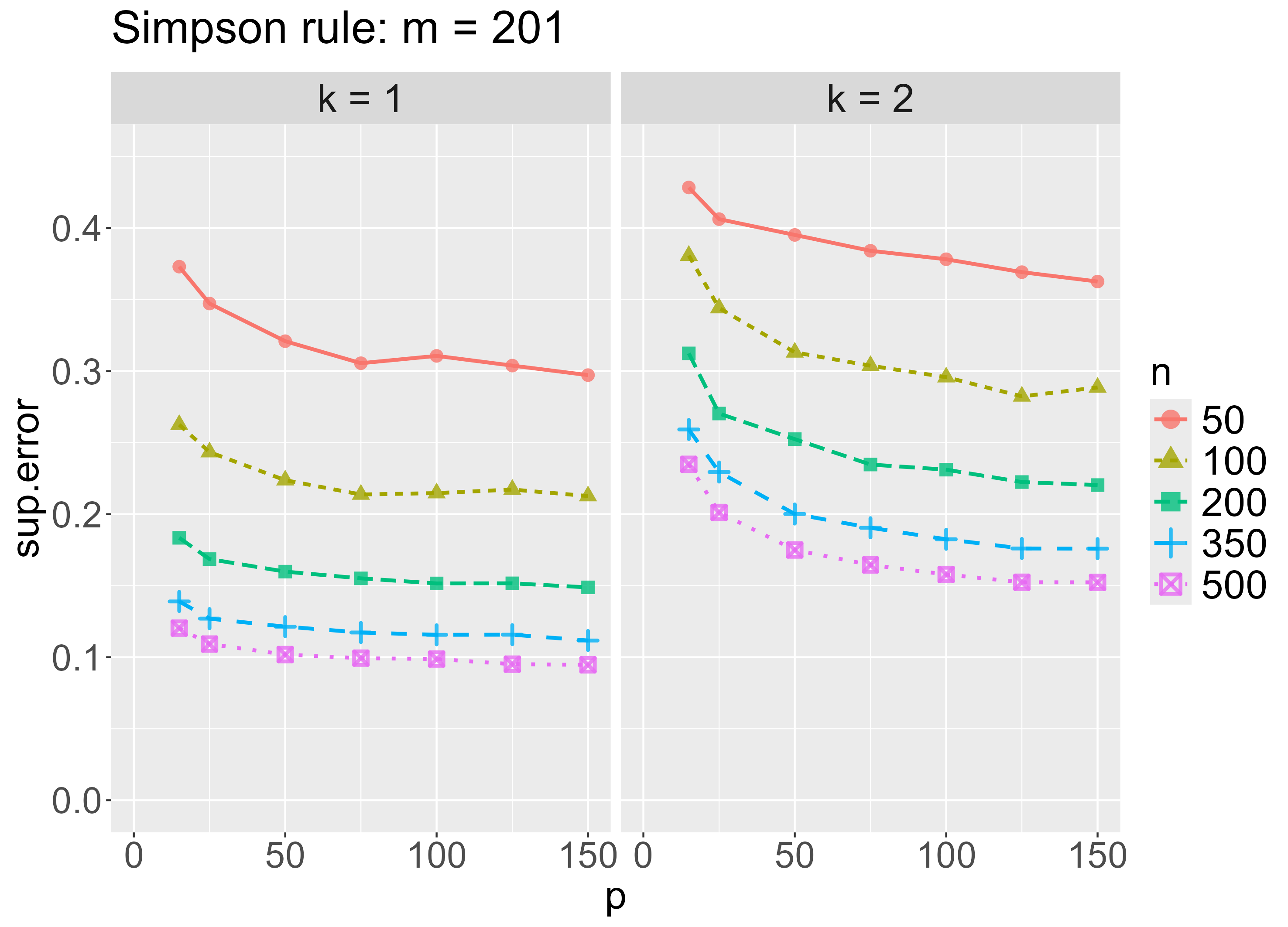}
    \caption{\small Mean of $\|\hat{\phi}_{k;n,201}(\cdotp;h^*) - \hat{c}_{k;201} \, \phi_k \|_{\infty}$.}
    \end{subfigure}
    \hfill
    \begin{subfigure}{0.49\linewidth}
    \includegraphics[width=\linewidth]{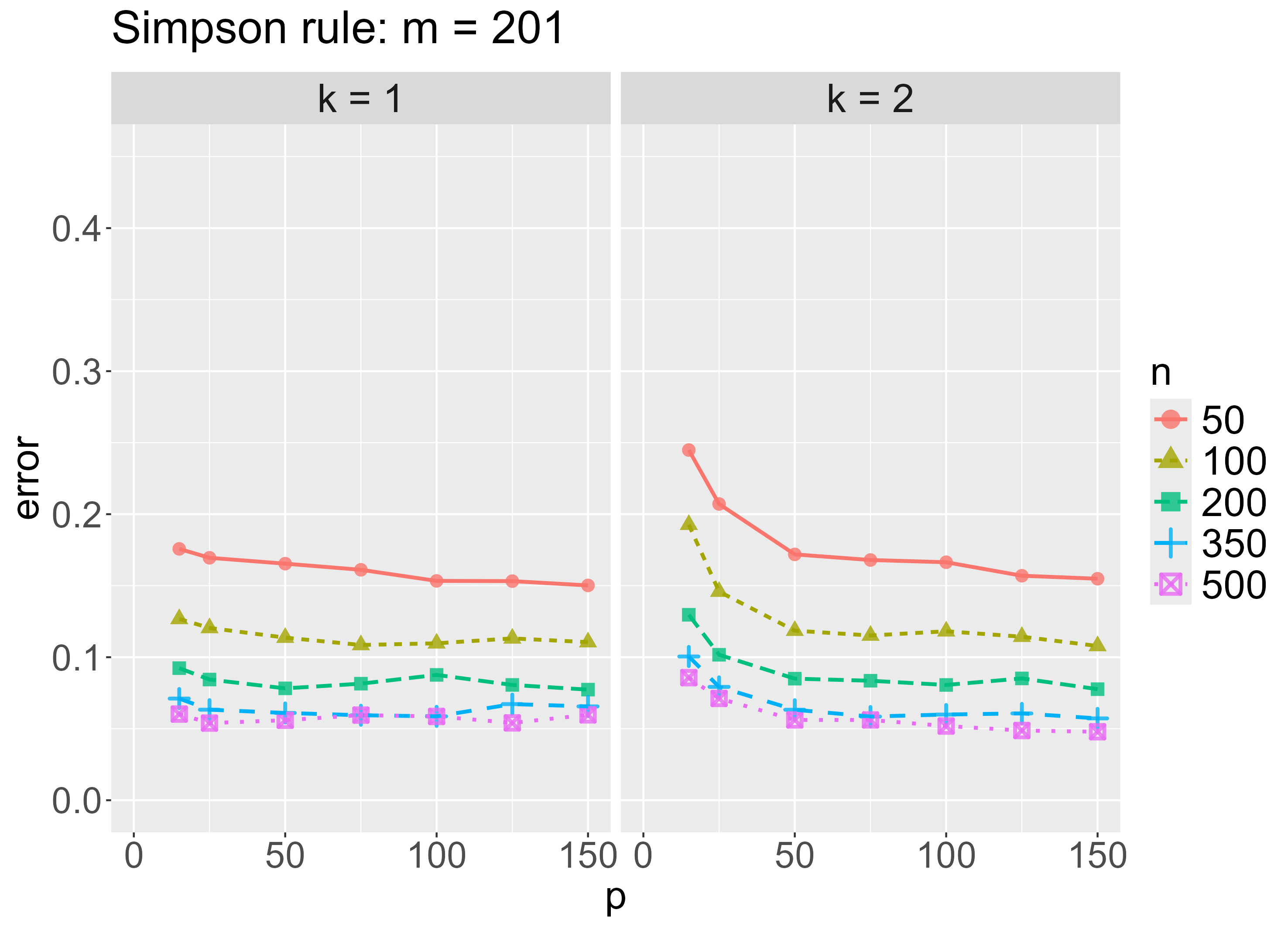}
    \caption{\small Mean of $|(\hat{\lambda}_{k;n,201}(h^*) - \lambda_{k})/\lambda_{k}|$.}
    \end{subfigure}
    \caption{Impact of increasing $p$, where for each combination of $n$ and $p$ the \rev{oracle bandwidth $h^* = h^*(n,p;k)$ is determined as in Figure \ref{fig: optimal slection}(a)}.}
    \label{fig: supnorm error}
\end{figure}

\subsection{Daily temperature curves in Berlin}\label{sec:realdata}

As an empirical illustration we revisit a dataset analyzed in \citet{wilk2026}, and consider time series of daily air temperatures for each month from $2000$ to $2025$ at a weather station in Berlin, Germany\footnote{data obtained from Deutsche\rev{r} Wetterdienst (DWD) at \\ \href{https://opendata.dwd.de/climate_environment/CDC/observations_germany/climate/10_minutes/air_temperature/historical/}{opendata.dwd.de/climate\_environment/CDC/observations\_germany/climate/10\_minutes/air\_temperature/historical/}}.
Thus we have $26$ segments for each of the $12$  time series covering the respective month. We focus on January and August, for which each segment has a length of $31$ days. \rev{After removing days with missing measurements, $n=806$ days remain for January and $n=803$ for August.} Measurements are taken every $10\,$$\,$minutes, leading to $145$ observations  and a dense grid of observations for each day.

We consider both the lag-0 as well as the long-run covariance kernel, and use the estimates from  \citet{wilk2026}\footnote{reproducible from the Github repository  \href{https://github.com/KevinWilk/Beyond-average-warming}{github.com/KevinWilk/Beyond-average-warming}}. \rev{These are local linear estimates, computed from the observations on an extended time window of $34$ hours, from 19:00 of the previous day to 05:00 of the following day. In their computation the empirical covariances are averaged over all pairs of time points which coincide modulo $24$ hours. Hence the kernel estimates are essentially periodic in each argument, and the estimated FPCs take practically the same values at 00:00 and at 24:00. For the long-run covariance kernel this periodicity also holds for the population kernel, since consecutive daily curves are contiguous in time. For the lag-0 covariance kernel it means that for time points close to midnight, covariances within a day are averaged with covariances between consecutive days, so that a periodized version of the lag-0 covariance kernel is estimated.} From these we derive estimates of the first three basis functions, and the first six eigenvalues, using  $97$ equidistant quadrature points corresponding to $15$-minute intervals and the composite Simpson rule. The results are displayed in Figure \ref{fig:lag0app} for the lag-0 covariance kernel and in Figure \ref{fig:longrunapp} for the long-run covariance kernel. Even though the cross-covariance kernels contribute substantially to the long-run covariance kernel, the estimated FPCs of lag-0 and long-run covariance kernels are \rev{quite} similar, with the first FPC corresponding to a parallel shift, and the second and third FPCs corresponding to night-afternoon deviations and morning - evening deviations. For interpretability we provide the square-roots of the eigenvalue estimates, which correspond to standard deviations in the Karhunen–Loève expansion. The first eigenvalue of the parallel shift is quite dominant, in particular in January.

\begin{figure}[h!]
    \centering
    \begin{subfigure}{0.5\linewidth}
    \includegraphics[width=\linewidth]{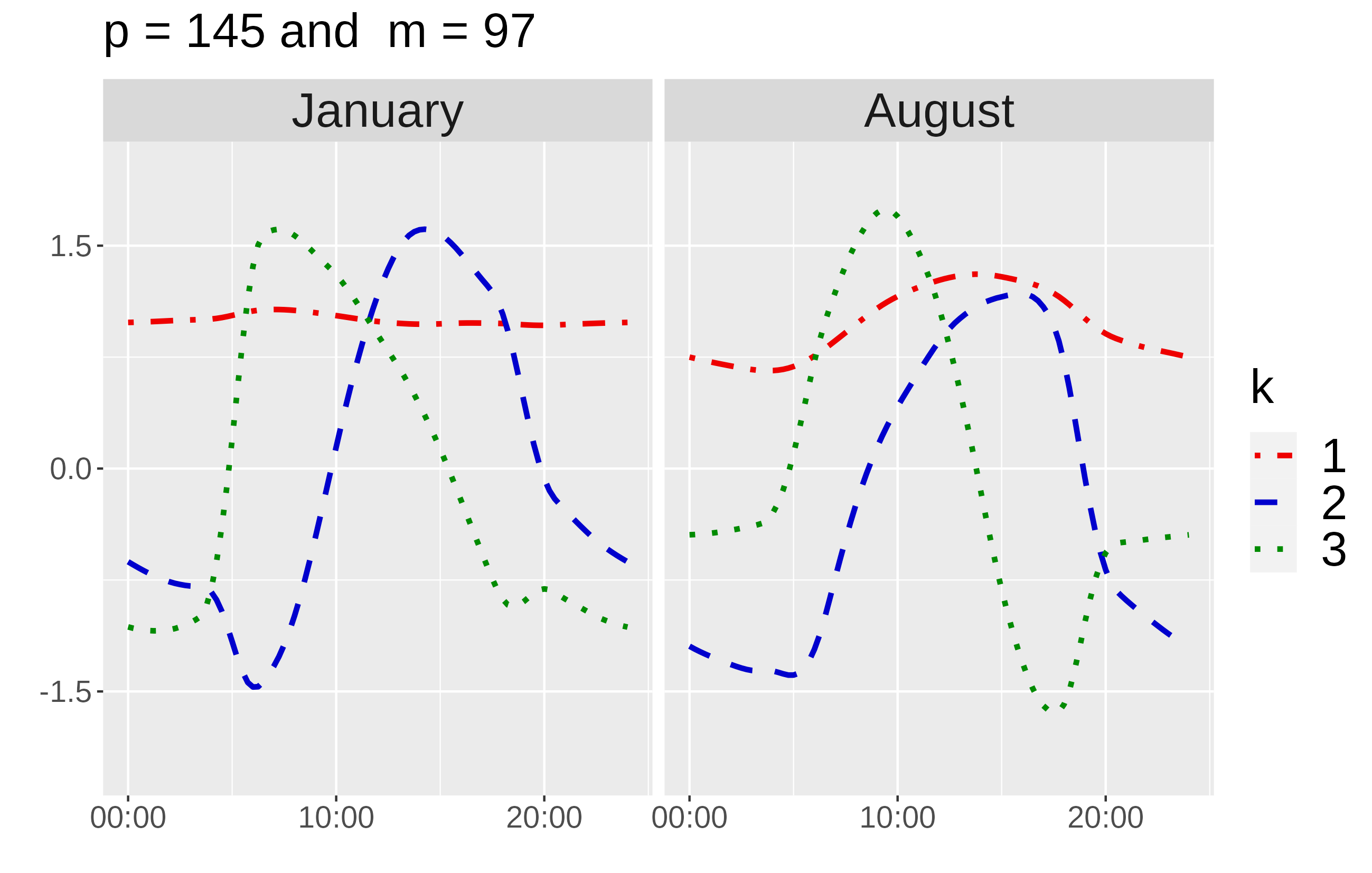}
    \end{subfigure}
    \hfill
    \raisebox{2cm}{
    \small
    \centering
    \begin{tabular}{c|c|cc}
                                     & January       &     August    &  \phantom{xxxxx}   \\
                                     & $n = 806$     &   $n = 803$   &     \\  \hline
     $\hat{\lambda}_{1;n,97}^{1/2}$  & $4{.}6442$   & $3{.}2399$   &     \\
     $\hat{\lambda}_{2;n,97}^{1/2}$  & $0{.}6776$   & $0{.}9124$  &      \\
     $\hat{\lambda}_{3;n,97}^{1/2}$  & $0{.}5671$   & $0{.}5789$   &      \\
     $\hat{\lambda}_{4;n,97}^{1/2}$  & $0{.}2669$   & $0{.}3408$   &      \\
     $\hat{\lambda}_{5;n,97}^{1/2}$  & $0{.}2576$   & $0{.}2663$   &     \\
     $\hat{\lambda}_{6;n,97}^{1/2}$  & $0{.}1801$   & $0{.}2305$   &     \\
    \end{tabular}
    }
    \caption{Estimation of FPCs and corresponding eigenvalues of the marginal lag-0 covariance kernels.}\label{fig:lag0app}
\end{figure}
\rev{Uniform confidence bands for the FPCs, based on the bootstrap in Remark~\ref{rem:bootstrap}, will be added in a later version of this paper.}

\begin{figure}[h!]
    \centering
    \begin{subfigure}{0.5\linewidth}
    \includegraphics[width=\linewidth]{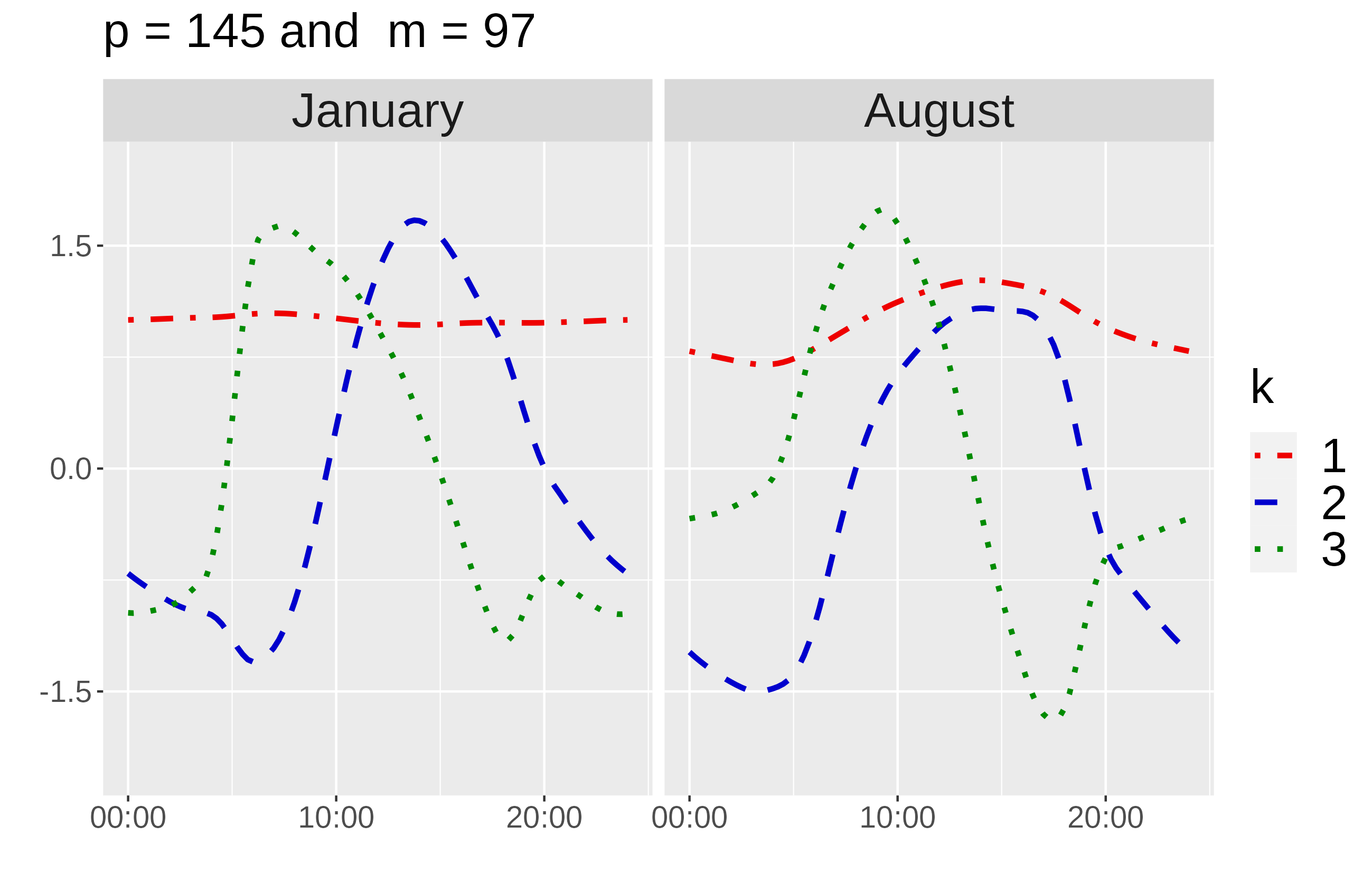}
    \end{subfigure}
    \hfill
    \raisebox{2cm}{
    \small
    \centering
    \begin{tabular}{c|c|cc}
                                     & January       &     August    &  \phantom{xxxxx}   \\
                                     & $n = 806$     &   $n = 803$   &     \\  \hline
     $\hat{\lambda}_{1;n,97}^{1/2}$  & $6{.}6354$   & $4{.}6500$    &     \\
     $\hat{\lambda}_{2;n,97}^{1/2}$  & $0{.}7793$    & $1{.}0519$  &     \\
     $\hat{\lambda}_{3;n,97}^{1/2}$  & $0{.}4823$   & $0{.}5650$   &     \\
     $\hat{\lambda}_{4;n,97}^{1/2}$  & $0{.}2911$   & $0{.}3746$   &     \\
     $\hat{\lambda}_{5;n,97}^{1/2}$  & $0{.}2540$   & $0{.}2895$    &     \\
     $\hat{\lambda}_{6;n,97}^{1/2}$  & $0{.}1683$   & $0{.}2441$   &     \\
    \end{tabular}
    }
    \caption{Estimates of the FPCs and corresponding eigenvalues for the long-run covariance kernels.}\label{fig:longrunapp}
\end{figure}

\section{Proofs} \label{sec:proofs}

\subsection{Proof of Theorem \ref{th:asympnorm}}\label{sec:proofasympnorm}

\begin{revblock}
\begin{proof}[Proof of Theorem~\ref{th:asympnorm}]

We start by showing that \eqref{eq:weakconvcovkernel} also holds in $C_s([0,1]^2)$: since $C_s([0,1]^2)$ is closed in $C([0,1]^2)$, contains $\sqrt n\,(\hgam_n-\Gamma)$, and closed subsets of $C_s([0,1]^2)$ are closed in $C([0,1]^2)$, the portmanteau theorem first gives $\prob(\mc G\in C_s([0,1]^2)) = 1$ for the limit process, and second also the weak convergence in $C_s([0,1]^2)$.

Turning to the remaining details for \textit{Step 1.}, denote $\mc T_k(G)  = \big(G, \Psi_k(G), \Phi_k(G)\big)$, where $\mc T_k \colon C_s([0,1]^2) \to C_s([0,1]^2)\times\R\times L_2[0,1]$, and
$$\mc T_k'(\Delta) = \Big(\Delta, \langle T_\Delta\phi_k,\phi_k\rangle, \sum_{j \neq k}(\lambda_k-\lambda_j)^{-1}\langle T_\Delta\phi_k,\phi_j\rangle\,\phi_j\Big).$$

The map $\mc T_k'$ is  linear and bounded by Bessel's inequality, and from the perturbation bounds in Lemma~\ref{lem:eigenexpansion} below,  for $\Delta \in C_s([0,1]^2)$ with $\|\Delta\|_\infty < \delta_k/4$ we have that
$$ \big\|\mc T_k(\Gamma + \Delta) - \mc T_k(\Gamma) - \mc T_k'(\Delta)\big\| \leq \Big(\frac{6}{\delta_k} + \frac{10}{\delta_k^2}\Big)\,\|\Delta\|_{\infty}^2,$$
where the norm on $C_s([0,1]^2)\times\R\times L_2[0,1]$ on the left is the sum of the norms on the three factors.

\smallskip

For \textit{Step 2,} denote $\chi(G,\lambda,g)  = \big(\lambda, \lambda^{-1}\,T_Gg\big)$  for $\lambda>0$, where
$\chi \colon C_s([0,1]^2)\times\R\times L_2[0,1] \to \R\times C[0,1]$, with an arbitrary (say constant) definition for $\lambda \leq 0$.

Using the mixed norm bound \eqref{eq:operatornormmixed} we get that $\|\mc L_k(\Delta,s,f)\|_\infty \leq \lambda_k^{-1}\|\Delta\|_\infty + |s|\,\lambda_k^{-2}\|\Gamma\|_\infty + \lambda_k^{-1}\|\Gamma\|_\infty\|f\|_{L_2}$, so that $\mc L_k$ in \eqref{eq:limit} is linear and bounded. Turning to $\chi$ and its second coordinate function,
for $|s|\leq\lambda_k/2$, using \eqref{eq:basicexpansion} we can write
\begin{align*}
    \, & \frac{T_{\Gamma + \Delta}(\phi_k + f)}{\lambda_k + s} - \frac{T_\Gamma\phi_k}{\lambda_k} - \mc L_k(\Delta,s,f)\\
    \, = \, & \frac{T_\Delta (f + \phi_k)}{\lambda_k + s} - \frac{T_\Delta \phi_k}{\lambda_k} - \Big(\frac{s}{(\lambda_k + s)\,\lambda_k}\, T_\Gamma ( f + \phi_k) - \frac{s}{\lambda_k^2}\, T_\Gamma \phi_k \Big)\\
    \, = \, & \frac{T_\Delta f }{\lambda_k + s} - \big(T_\Delta \phi_k + T_\Gamma f \big)\, \frac{s}{(\lambda_k + s)\,\lambda_k}\, + \, \frac{s^2}{(\lambda_k + s)\,\lambda_k^2}\, T_\Gamma \phi_k
\end{align*}
With $\|\phi_k\|_{L_2}=1$ and $\lambda_k + s\geq\lambda_k/2$, the supremum norm of the right-hand side is at most
$$\frac{2\, \|\Delta\|_\infty\|f\|_{L_2}}{\lambda_k} + \frac{2\, |s|}{\lambda_k^2}\,\big(\|\Delta\|_\infty + \|\Gamma\|_\infty\|f\|_{L_2}\big) + \frac{2\,s^2}{\lambda_k^3}\,\|\Gamma\|_\infty,$$
which is $\mathcal O\big((\|\Delta\|_\infty + |s| + \|f\|_{L_2})^2\big)$.

\smallskip

To conclude, by \eqref{eq:eigenfctexpress}, for $\|G-\Gamma\|_\infty\leq\delta_k/4$ the composite $\chi\circ\mc T_k$ of the maps from Steps 1 and 2 is the map $G\mapsto\big(\Psi_k(G),\Phi_k(G)\big)$, with the second component now as element in $C[0,1]$. Hence by the chain rule for Fr\'echet derivatives, $G\mapsto\big(\Psi_k(G),\Phi_k(G)\big)$ is Fr\'echet differentiable at $\Gamma$ as a map from $C_s([0,1]^2)$ into $\R\times C[0,1]$.

Further, on the event $\{\|\hgam_n-\Gamma\|_\infty\leq\delta_k/4\}$, the probability of which tends to one by \eqref{eq:weakconvcovkernel}, Lemma~\ref{lem:eigenexpansion} shows that $\hat\lambda_{k;n}>0$ is a simple eigenvalue of $\hat T_n$, and that $|\langle\hat\phi_{k;n},\phi_k\rangle|\geq 3/4$. Hence on this event, by the choice of the sign in the theorem, $\big(\hat\lambda_{k;n},\hat\phi_{k;n}\big) = \big(\Psi_k(\hgam_n),\Phi_k(\hgam_n)\big)$. Since changing random elements on events whose probabilities tend to zero does not affect convergence in distribution, the functional delta method \citep[Theorem 3.9.4]{van1996weak} then actually yields joint convergence in distribution of
$\sqrt n\,\big(\hat\lambda_{k;n}-\lambda_k,\,\hat\phi_{k;n}-\phi_k\big)$ in $\R\times C[0,1]$, with an $\mathcal O_{\mathbb P}(n^{-1/2})$-remainder term, and the weak limit of  $\sqrt n\,\big(\hat\phi_{k;n}-\phi_k\big)$ in $C[0,1]$ must coincide with that in $L_2$ in \eqref{eq:weakinL2}, since weak convergence in $C[0,1]$ implies weak convergence in $L_2[0,1]$.

This concludes the proof of the theorem.
\end{proof}
\end{revblock}

\begin{revblock}
For convenience we provide the following explicit simple perturbation bounds, see \citet[Theorems 5.1.6 and 5.1.8]{hsing2015theoretical} for more sophisticated and precise bounds.
\begin{lemma}\label{lem:eigenexpansion}
Suppose that $\lambda_k>0$ is a simple eigenvalue of $T_\Gamma$, so that $\delta_k>0$.
Let $\Delta\in L_{2,s}([0,1]^2)$ with $\|\Delta\|_{L_2([0,1]^2)} \leq \delta_k/4$, and let $\lambda_{k,\Gamma+\Delta}$ be the $k$-th eigenvalue of $T_{\Gamma+\Delta}$ in the convention of Section~\ref{sec:perturbationtheory}. Then $\lambda_{k,\Gamma+\Delta} \geq 3\,\delta_k/4 > 0$ is a simple eigenvalue of $T_{\Gamma+\Delta}$, and the normalized eigenfunctions $\pm \phi_{k, \Gamma+\Delta} $ of $T_{\Gamma+\Delta}$ for $\lambda_{k,\Gamma+\Delta}$ satisfy $|\langle \phi_{k, \Gamma+\Delta} , \phi_k\rangle|\geq 3/4$. If the sign is chosen such that $\langle \phi_{k, \Gamma+\Delta},\phi_k\rangle>0$, then
\begin{align}
\big|\lambda_{k,\Gamma+\Delta} - \lambda_k - \langle T_\Delta\phi_k,\phi_k\rangle\big| & \leq \frac{6\,\|\Delta\|_{L_2([0,1]^2)}^2}{\delta_k}, \label{eq:expeigval}\\
\Big\|\phi_{k, \Gamma+\Delta} - \phi_k - \sum_{j\neq k}\frac{\langle T_\Delta\phi_k,\phi_j\rangle}{\lambda_k-\lambda_j}\,\phi_j\Big\|_{L_2} & \leq \frac{10\,\|\Delta\|_{L_2([0,1]^2)}^2}{\delta_k^2}. \label{eq:expeigfct}
\end{align}
\end{lemma}

\rev{The proof of Lemma~\ref{lem:eigenexpansion} is given in Appendix~\ref{sec:appendixperturbation}.}
\end{revblock}

\subsection{Proofs of Theorems \ref{theorem:optimalityeigenval} and  \ref{theorem:optimality}}\label{ssec:proof:optimality}

\begin{proof}[Proof of Theorem \ref{theorem:optimalityeigenval}]

    \rev{By the reduction to two hypotheses \citep[Theorem 2.2]{tsybakov2008introduction}, for each of the two rates it suffices to construct two hypotheses in $\mc P_{\mathcal G}(\gamma;L)$ for which the $k$-th eigenvalues differ by at least a constant multiple of the rate, while the Kullback-Leibler divergence between the distributions of the observations remains bounded.}
\rev{Here, for clarity of exposition we provide the construction under the additional assumption of a regular design: for a constant $\cd>0$ independent of $n$ and $p$ we assume $x_{j+1} - x_j = x_{j+1}(p,n) - x_j(p,n) \geq (\cd\, p)^{-1}$.  In the appendix, Section \ref{sec:appendixperturbation} we show that \eqref{eq:lowerboundeigenvalue} remains valid for general design.}

    First consider $k=1$.
    We start with the lower bound $p^{-\gamma}$ in \eqref{eq:lowerboundeigenvalue}.
Take $\phi_{1;0}=1$, $\lambda_{1;0}=1$, $Z_{i;0} = W_i\, \lambda_{1;0}^{1/2}\, \phi_{1;0} = W_i$, where $W_i \sim \mathcal{N}(0,1)$, which are taken independent over $i$.
Further, using the additional design assumption, setting $\varphi(x) = \exp(-(1-x^2)^{-1})\, 1_{|x|<1}$ we consider
    \begin{equation}\label{eq:testfct}
	\varphi(x;p) = \sum_{l=1}^{p-1}\, p^{- \gamma}\,
	\varphi\big((2\, p \, \cd)\, (x -(x_{l}+ x_{l+1})\,/\,2)\big),
    \end{equation}
   and
     $$\lambda_{1;1,p} = \int \big(1+\varphi(\cdot;p)\big)^2 ,\qquad \phi_{1;1,p}(x) = \frac{1+\varphi(x;p)}{\lambda_{1;1,p}^{1/2}},$$
     and finally
     $$ Z_{i;1,p} = W_i\, \lambda_{1;1,p}^{1/2}\, \phi_{1;1,p}. $$
    By construction, $\int \phi_{1;1,p}^2 =1$, so that indeed $\phi_{1;1,p}$ is the principal component basis function and $\lambda_{1;1,p}$ the associated eigenvalue. Further, from the design assumption, $\varphi(x_j;p)=0 $ at all design points $x_j$, hence the distributions of the observations resulting from $Z_{i;0}$ and $Z_{i;1,p}$ are the same.
    Furthermore, using the fact that the terms in \eqref{eq:testfct} have disjoint supports, we obtain
    \begin{align*}
        \lambda_{1;1,p} - 1 & = \int \big(2\,  \varphi(\cdot;p) + \varphi(\cdot;p)^2\big)\\
                & \geq  \frac1{2\, \, \cd}\, (p-1)\, p ^{-\gamma-1}\int \varphi
                 \,\geq \, c\, p^{-\gamma}
    \end{align*}
     for some $c>0$, as required.
\rev{Finally, for a constant $\tilde L>0$ we replace $Z_{i;0}$ and $Z_{i;1,p}$ by $\tilde L\, Z_{i;0}$ and $\tilde L\, Z_{i;1,p}$. This leaves the eigenfunctions unchanged, and the observations under both hypotheses still have the same distribution. The covariance kernels become $\tilde L^2$ and $\Gamma_{1;p}(x,y) = \tilde L^2\,(1+\varphi(x;p))\, (1+\varphi(y;p))$, with first eigenvalues $\tilde L^2$ and $\tilde L^2\,\lambda_{1;1,p}$, which differ by at least $\tilde L^2\, c\, p^{-\gamma}$. Since the summands in \eqref{eq:testfct} have disjoint supports of length of order $p^{-1}$ and height $p^{-\gamma}$, the H\"older norm of $(1+\varphi(x;p))\, (1+\varphi(y;p))$ is bounded uniformly in $p$, so that both kernels are contained in $\mc H(\gamma,L)$ for sufficiently small $\tilde L$.}

    \vspace{0.1cm}

	Next consider the lower bound of order $n^{-1/2}$. Here  we can argue as for the covariance kernel in \citet{berger2024optimal} by considering $Z_{i;0} = W_i$ as above and $Z_{i;1,n} = \sigma_{1;n}\, W_i$, $\sigma_{1;n} = 1+ n^{-1/2}$, so that the basis function is equal to $1$ in both cases, and the eigenvalues $1$ and $\sigma_{1;n}^2$ differ by at least $n^{-1/2}$ as required. For the bound on the Kullback-Leibler divergence of the observations see  \citet{berger2024optimal}. \rev{As above, we multiply both hypotheses by a sufficiently small constant $\tilde L>0$, so that their covariance kernels $\tilde L^2$ and $\tilde L^2\,\sigma_{1;n}^2$ are contained in $\mc H(\gamma,L)$, and their eigenvalues still differ by at least $\tilde L^2\, n^{-1/2}$. This scaling only changes the error variance: after dividing the observations by $\tilde L$, the scaled hypotheses with errors of variance $\sigma_0^2$ correspond to the unscaled ones with errors of variance $\sigma_0^2/\tilde L^2$, and the bound on the Kullback-Leibler divergence holds for any error variance.}

\bigskip

\rev{Next let us extend these arguments to general $k$. Choose $k-1$ smooth (for example trigonometric), normalized and pairwise orthogonal functions $\psi_1, \ldots, \psi_{k-1}$ on $[0,1]$, independent of $n$ and $p$, which are orthogonal to constants (that is, integrate to $0$).  For $W_{i,j} \sim \mathcal N(0,1)$, $i=1, \ldots, n$, $j=1, \ldots, k$, all independent, set $V_i = \sqrt{2}\, (W_{i,1}\, \psi_1 + \ldots + W_{i,k-1}\, \psi_{k-1})$ and $Z_{i;0,k} = V_i + W_{i,k}$. The covariance kernel $\Gamma_0$ of $Z_{i;0,k}$ has eigenvalues $2$ (with multiplicity $k-1$), $1$ and $0$, so that its $k$-th eigenvalue is $1$ with spectral gap $\delta_k = 1$.

For the lower bound of order $n^{-1/2}$ set $Z_{i;1,n,k} = V_i + \sigma_{1;n}\, W_{i,k}$, with $k$-th eigenvalue $\sigma_{1;n}^2$ for sufficiently large $n$. The distributions of $Z_{i;0,k}$ and $Z_{i;1,n,k}$ are convolutions of those of $Z_{i;0}$, respectively $Z_{i;1,n}$ from the above argument with the same Gaussian measure, namely the distribution of $V_i$. Since by the data processing inequality the Kullback-Leibler divergence between two distributions only decreases if both are convolved with the same distribution, the argument extends from $k=1$ to general $k$.}

\rev{For the lower bound of order $p^{-\gamma}$ set $Z_{i;1,p,k} = V_i + \big(1+\varphi(\cdot;p)\big)\, W_{i,k}$ with $\varphi(\cdot;p)$ from \eqref{eq:testfct}. Since $\varphi(x_j;p)=0$, the distributions of the observations at the design points coincide. The covariance kernel of $Z_{i;1,p,k}$ equals $\Gamma_0 + \Delta_p$ with $\Delta_p(x,y) = \varphi(x;p) + \varphi(y;p) + \varphi(x;p)\,\varphi(y;p)$, where $\|\Delta_p\|_{L_2([0,1]^2)} = \Oop{p^{-\gamma}}$. Hence by Lemma~\ref{lem:eigenexpansion}, for sufficiently large $p$ its $k$-th eigenvalue can be expanded as $1 + \langle T_{\Delta_p} 1, 1\rangle + \Oop{p^{-2\gamma}} = 1 + 2\int\varphi(\cdot;p) + \Oop{p^{-2\gamma}}$, and by the computation for $k=1$ it differs from $1$ by at least $c\,p^{-\gamma}$. Note that we do not require the $\psi_j$ to be orthogonal to $\varphi(\cdot;p)$. Finally, as for $k=1$ all hypotheses are multiplied by a sufficiently small constant $\tilde L$, so that the covariance kernels are contained in $\mc H(\gamma,L)$.}

\end{proof}

For the proof of Theorem \ref{theorem:optimality} we need to modify the hypotheses functions from the proof of  \citet[Theorem 2]{berger2024optimal}.
To construct normalized basis functions with $\int \phi^2 = 1$ in the hypotheses we use the following lemma.
\begin{lemma}\label{lem:constrfct}
\begin{enumerate}
\item    Consider $\varphi(x) = \exp(-(1-x^2)^{-1})\, 1_{|x|<1}$ for which $\varphi^\prime(x) = - \frac{2x}{(1-x^2)^2}\, \exp(-(1-x^2)^{-1})\, 1_{|x|<1}$ and let $\varphi_w(x) = w\, \varphi^\prime(x) - (1-w)\, \varphi(x)   $, $w \in [0,1]$. For all sufficiently small $h>0$ there is a $w = w(h)$ such that
        \begin{equation}\label{eq:solveint}
            \int \varphi_w \, +\, h\,\int \varphi_w^2 = 0.
        \end{equation}
    \rev{Moreover}, $\varphi_w(1/2) \leq -\exp(-4/3)$.
    \item Given sufficiently small $h_0, h_1 >0$ and $x_0 \in (0,1)$ so that $[x_0-h_1,x_0 + h_1]\subseteq [0,1]$, choosing $w=w(h_0/2)$ as in part i) we obtain
    \begin{equation}\label{eq:normalization}
        \int_0^1 \big(1 + h_0\, \varphi_w\big((x-x_0)/h_1 \big) \big)^2\, \dx x = 1.
    \end{equation}
\end{enumerate}
\end{lemma}
\begin{proof}[Proof of Lemma \ref{lem:constrfct}]
    For i): Since $\varphi$ is even, $\varphi^\prime$ and the product $\varphi\, \varphi^\prime$ are odd, hence integrate to zero. Thus we have $\{\int \varphi_w \mid w \in [0,1]\}= [- \int \varphi, 0] $, and
    $$ \int \varphi_w^2 = (1-w)^2\, \int \varphi^2 + w^2 \, \int \big(\varphi^\prime\big)^2,$$
    which also depends continuously on $w$. Setting $c_1 =  \int \varphi^2$, $c_2 = \int \big(\varphi^\prime\big)^2$, we have that $\min(c_1, c_2)>0$, $\max(c_1,c_2) < \infty$ and
    $$ $$
    $$ \{\int \varphi_w \, + \, h\,\int \varphi_w^2  \mid w \in [0,1]\}\supseteq  \big[- \int \varphi  + h\, c_1,  h\, c_2\big], $$
    so that the existence of a solution to  \eqref{eq:solveint} for sufficiently small $h>0$ follows from the intermediate value theorem.

    Further $\varphi(1/2) = \exp(-4/3)$, while $\varphi^\prime(1/2) < - \exp(-4/3)$, from which the claim is then evident.

    For ii):
    \begin{align*}
        \int_0^1 \big(1 + h_0\, \varphi_w\big((x-x_0)/h_1 \big) \big)^2\, \dx x = 1 + \, 2\, h_0\, h_1\, \Big(\int \varphi_w + \frac{h_0}{2} \int \varphi_w^2\Big),
    \end{align*}
    from which the claim follows from i) upon the choice of $w = w(h_0/2)$.
\end{proof}

\begin{proof}[Proof of Theorem \ref{theorem:optimality}, case $k=1$]

	In the proof we rely on the reduction to hypothesis testing as presented e.g.~in \citet[section 2]{tsybakov2008introduction}.
	For  all hypothesis models we take $\mu=0$. The arguments are inspired by those from \citet{berger2024optimal}, but modifications are required.

    We present the proofs for $k=1$ for the lower bounds of order $  p^{-\gamma}$ and $(\log(n\, {p})/(n{  p}))^{\gamma/(2\gamma+1)}$. \rev{Extensions to general $k$ along the lines of  the proof of Theorem  \ref{theorem:optimalityeigenval} are given in the Appendix, Section \ref{sec:appendixperturbation}. Further, the lower bound of order $n^{-1/2}$ can be obtained as in \citet[Proposition 2]{belhakem2025minimax} for $k=1$ and in \citet{bourarach2026minimax} for general $k$, where two FPCs are rotated by an angle of order $n^{-1/2}$. In the Appendix, Section \ref{sec:appendixperturbation}, we give a short proof along these lines for general design points, which also applies to continuously observed processes without noise.}

    First consider the lower bound $  p^{-\gamma}$. We construct the hypotheses at appropriate distance so that the observations are equal at all design points.  To this end we take $\phi_{1;0}=1$, $Z_{i;0} = W_i\, \phi_{1;0} = W_i$, where $W_i \sim \mathcal{N}(0,1)$, which are taken independent over $i$.
    For the competing hypotheses, set $x_0 = 0, x_{p+1}=1$. There is an $l \in \{0,1,\ldots, p\}$ for which $x_{l+1}- x_{l}\geq (p+1)^{-1}$, choose such an $l$.
   We then set
	$$ \phi_{1;1,p}  = 1 + p^{-\gamma}\, \varphi_w\big( 2\, (p+1)\, (x -(x_{l}+ x_{l+1})\,/\,2)\big),$$
    where $\varphi_w$ and $w = w(p^{-\gamma}/2)$ are chosen according to Lemma \ref{lem:constrfct}. Using Lemma \ref{lem:constrfct}, ii), $\int_0^1 \phi_{1;1,p}^2=1$, and we set 	$ Z_{i;1,p}(  x)= W_i\, \phi_{1;1,p}(x).$ Since $\phi_{1;1,p}(x_j)=1$ at all design points $x_j$ and hence $Z_{i;1,p}(  x_j)= W_i$, so that the distributions under both hypotheses coincide.
    For sufficiently large $p$,
    $$ \dd\big(\phi_{1;1,p}, \phi_{1;0}\big)_{\infty,s} \geq p^{-\gamma}\, |\varphi_w(1/2)| \geq \exp(-4/3) \, p^{-\gamma}, $$
    as required. \rev{Finally, as in the proof of Theorem~\ref{theorem:optimalityeigenval} we multiply both hypotheses by a constant $\tilde L>0$. This leaves the eigenfunctions and the equality of the distributions of the observations unchanged. The covariance kernels become $\tilde L^2$ and $\Gamma_{1,p}(x,y) = \tilde L^2\,\phi_{1;1,p}(x) \, \phi_{1;1,p}(y)$; both have rank one, so that their spectral gap is $\delta_1 = \tilde L^2$. Since the H\"older norm of $\phi_{1;1,p}(x) \, \phi_{1;1,p}(y)$ is bounded by a constant $C$ independent of $p$, choosing $\tilde L^2 = L/C$ shows that both hypotheses belong to $\mc P_{\mathcal G}(\gamma;1,L)$ if $\delta \leq L/C$.}
    \vspace{0.3cm}

    Second let us turn to the lower bound of order $(\log(n\, {p})/(n{  p}))^{\gamma/(2\gamma+1)}$. Here we can use the construction from \citet{berger2024optimal}. Setting $N_{n,  p}= \ceil{c_0\,(np/(\log(n {p})))^{1/(2\gamma+1)}}$ for $c_0>0$ sufficiently small, $ h_{n,  p} = N_{n,  p}^{-1}$ and $z_l = (l-1/2)/N_{n,  p}$,
    we only need to replace the hypothesis functions defined in \citet{berger2024optimal} as
    $$ 1 +  \kappa\, (h_{n,  p}\,/\,2)^{\gamma} \varphi\big(2\, (  x -  z_l)/h_{n,  p}\big),$$
    where $\varphi$ is as in Lemma \ref{lem:constrfct} and $\kappa>0$ is a sufficiently small constant, by
    $$ 1 +  \kappa\, (h_{n,  p}\,/\,2)^{\gamma} \varphi_w\big(2\, (  x -  z_l)/h_{n,  p}\big),$$
    with $w = w(\kappa\,(h_{n,  p}\,/\,2)^{\gamma}/2)$ in Lemma \ref{lem:constrfct} to guarantee that these are normalized basis functions. The rest of the argument remains the same by noting that for sufficiently large $np$, the sup-norm distances in \citet{berger2024optimal} between the hypothesis functions coincide with the semimetric $\dd()_{\infty,s}$. \rev{Finally, all hypotheses are multiplied by a constant $\tilde L>0$ as in the first part; as explained in the proof of Theorem~\ref{theorem:optimalityeigenval}, this only changes the error variance, and the argument remains valid.}

\end{proof}

\subsection{Proof of Theorem \ref{thm:rates_principlecomp_estimation}}\label{sec:upperboundsproofs}

\begin{proof}[Proof of Theorem \ref{thm:rates_principlecomp_estimation}]

(i) This follows from the perturbation bound \eqref{eq:spectraleigenvalues} and the upper bounds of covariance kernel estimation in $L_2$ in  Lemma \ref{thm:rates_covkernrates}, part (i), in the Appendix, Section \ref{sec:appendix1}.

\smallskip

(ii) From (i) we obtain in particular that $\lamkest h - \lambda_k =  o_{\prob;  \, Z \in \Zclass}(1)$.
Observe that for $Z \in \Zclassk$, $\lambda_k \geq \delta_k \geq \delta >0$, and therefore $|\lamkest h|^{-1} = \mc O_{\prob;  \, Z \in \Zclassk}(1)$. \rev{In particular, $\prob_Z(\lamkest{h^\star} > 0) \to 1$ uniformly in $Z\in\Zclassk$, so that the convention for $\hat\phi_k$ in Section~\ref{sec:perturbationtheory} does not affect the rate.}
Hence the upper bound follows from \rev{\eqref{eq:supboundeigenfct}} and the upper bounds of the covariance kernel estimate in the supremum norm in Lemma \ref{thm:rates_covkernrates}, part (ii).

\end{proof}

\section*{Acknowledgements}

 HH and KW gratefully acknowledge financial support from the DFG, grant HO 3260/9-1. \rev{Parts of the results of this paper were presented in the online FDA\&A seminar (Functional Data Analysis and Applications) of the ANR project FUNMathStat on November 27, 2025. We thank the organizers for the invitation and the participants for helpful discussions.}

\rev{\textit{Use of AI tools.} The main results and proofs of the first version of this paper were developed in 2025. The 2025 free version of ChatGPT (OpenAI) was used as a discussion tool: for background in functional analysis; for discussing the overall approach to supremum-norm bounds based on the expansion \eqref{eq:basicexpansion}; and for discussing formulations used in the paper. For the revised version (2026), we used Claude Opus 5.5 and Fable 5.2 (Anthropic) as follows. \emph{Proof development:} to check the proofs and to propose, discuss and draft corrected arguments, in particular for the Fr\'echet differentiability in Remark \ref{rem:proofthasympnorm} on the continuous symmetric kernels, used in the proof of Theorem~\ref{th:asympnorm};  for the extension of the lower bounds in Theorems \ref{theorem:optimalityeigenval} and ~\ref{theorem:optimality} to general $k$ in order to guarantee Hölder-continuity of the covariance kernel; for formulating the construction of the lower bound of order $n^{-1/2}$ along the lines of \citet{bourarach2026minimax} in Appendix~\ref{sec:appendixperturbation}; and for the simplified perturbation bounds in Lemma \ref{lem:eigenexpansion} which replace the reference to \citet[Theorems 5.1.6 and 5.1.8]{hsing2015theoretical}. \emph{Writing and exposition:} to suggest and draft revisions of the text, in particular: Discussion of perturbation theory of symmetric, compact but not necessarily positive-semidefinite operators in Section \ref{sec:perturbationtheory}; to search for and check references, and to do edits to the LaTeX source. \emph{Other tasks:} to review and suggest improvements to parts of the code for the numerical experiments. All AI-assisted material was checked and revised by the authors, who take full responsibility for the content of the paper.}

\vspace{0.5in}

\bibliographystyle{Shortcuts/abbrvnat-doi}

\bibliography{Literature_fda}

\newpage

\appendix

\begin{revblock}
\section{Further proofs of technical results }\label{sec:appendixperturbation}

\begin{proof}[Proof of Lemma~\ref{lem:eigenexpansion}]
Let $G = \Gamma+\Delta$. In the following we write $\|\Delta\|_{L_2([0,1]^2)} = \|\Delta\|$. First we show that $\lambda_{k,G}\geq 3\,\delta_k/4>0$ is a simple eigenvalue of $T_G$ with spectral gap of at least $\delta_k/2$. Since $\lambda_k\geq\delta_k$, Weyl's inequality \eqref{eq:spectraleigenvalues} together with \eqref{eq:l2operator} gives $\lambda_{k,G}\geq\lambda_k-\|\Delta\|\geq 3\,\delta_k/4$. Further, for $j \not=k$ we have that
\begin{align*}
\big|\lambda_{k,G} - \lambda_{j,G}\big| & \geq \big|\lambda_{k} - \lambda_{j}\big| - \big|\lambda_{k,G} - \lambda_k\big| - \big|\lambda_{j,G} - \lambda_{j}\big| \\
& \geq \delta_k - 2\, \|\Delta\| \geq \delta_k/2,
\end{align*}
where we used Weyl's inequality \eqref{eq:spectraleigenvalues} and the assumption $\|\Delta\| \leq \delta_k/4$. Similarly we also get for $j\neq k$ that
\begin{equation}\label{eq:compareeigen}
|\lambda_{k,G} - \lambda_j|\geq \delta_k - \|\Delta\| \geq 3\,\delta_k/4.
\end{equation}
 Next, by the sine-$\theta$ theorem, \eqref{eq:spectraleigenfct} we have that $\|\phi_{k,G} - c\,\phi_k\|_{L_2}\leq 2\sqrt2\,\|\Delta\|/\delta_k \leq 1/\sqrt 2$ for $c = \sign\langle\phi_{k,G},\phi_k\rangle$, so that $c\,\langle \phi_{k,G},\phi_k\rangle = 1 - \|\phi_{k,G} -c\,\phi_k\|_{L_2}^2/2 \geq 3/4$.

Now choose the sign such that $\langle\phi_{k,G} ,\phi_k\rangle>0$. Analogously to the previous calculation,
\begin{equation}\label{eq:uppboundnorm}
0 \leq 1-\langle \phi_{k,G},\phi_k\rangle = \|\phi_{k,G}-\phi_k\|_{L_2}^2/2\leq 4\,\|\Delta\|^2/\delta_k^2
\end{equation}

Then, taking inner products of $T_G \phi_{k,G} = \lambda_{k,G}\,\phi_{k,G}$ with $\phi_j$ and using that $T_\Gamma$ is self-adjoint gives
\begin{equation}\label{eq:eigenidentity}
(\lambda_{k,G} - \lambda_j)\,\langle\phi_{k,G} ,\phi_j\rangle = \langle T_\Delta\phi_{k,G},\phi_j\rangle.
\end{equation}
To prove \eqref{eq:expeigval},
\eqref{eq:eigenidentity} for $j=k$ gives
\begin{align*}
 \lambda_{k,G} -\lambda_k - \langle T_\Delta\phi_k,\phi_k\rangle \, = \, & \frac{\langle T_\Delta \phi_{k,G},\phi_k\rangle -\langle \phi_{k,G},\phi_k\rangle\,\langle T_\Delta\phi_k,\phi_k\rangle}{\langle \phi_{k,G},\phi_k\rangle}\\
\, = \, &  \frac{\langle T_\Delta(\phi_{k,G}-\phi_k),\phi_k\rangle + \big(1-\langle \phi_{k,G},\phi_k\rangle\big)\,\langle T_\Delta\phi_k,\phi_k\rangle}{\langle \phi_{k,G},\phi_k\rangle},
\end{align*}
which, using \eqref{eq:uppboundnorm} in the second term, is upper bounded in absolute value by
\begin{align*}
& \, \frac{4 }{3}\, \| T_\Delta (\phi_{k,G}-\phi_k)\|_{L_2} + \frac{16 \, \|\Delta\|^2}{3\, \delta_k^2}\,\| T_\Delta \phi_k\|_{L_2}\\
\leq & \, \frac{4 }{3}\, \| \Delta\| \| \phi_{k,G}-\phi_k\|_{L_2} + \frac{16 \, \|\Delta\|^3}{3 \, \delta_k^2} \leq  \, \frac{8 \, \sqrt2}{3\, \delta_k}\, \| \Delta\|^2  + \frac{4 \, \|\Delta\|^2}{3 \, \delta_k}
\end{align*}

upon using \eqref{eq:spectraleigenfct} and $\|\Delta\|\leq \delta_k/4$ in the last step.

Turning to \eqref{eq:expeigfct}, for $j\neq k$ we divide by $\lambda_{k,G}-\lambda_j$ in \eqref{eq:eigenidentity}, and expanding $\phi_{k,G}$ in the basis $(\phi_j)$ yields
\begin{align*}
& \phi_{k,G} - \phi_k - \sum_{j\neq k}\frac{\langle T_\Delta\phi_k,\phi_j\rangle}{\lambda_k-\lambda_j}\,\phi_j \\
\, = \, & \sum_{j\neq k}\Big(\frac{\langle T_\Delta \, \phi_{k,G},\phi_j\rangle}{\lambda_{k,G}-\lambda_j} - \frac{\langle T_\Delta\phi_k,\phi_j\rangle}{\lambda_k-\lambda_j}\Big)\,\phi_j + \langle\phi_{k,G},\phi_k\rangle \,\phi_k - \phi_k \\
\, = \, & \sum_{j\neq k}\Big(\frac{\langle T_\Delta (\phi_{k,G}-\phi_k),\phi_j\rangle}{\lambda_{k,G}-\lambda_j} + \frac{(\lambda_k-\lambda_{k,G})\,\langle T_\Delta\phi_k,\phi_j\rangle}{(\lambda_{k,G}-\lambda_j)(\lambda_k-\lambda_j)}\Big)\,\phi_j - \big(1-\langle\phi_{k,G},\phi_k\rangle\big)\,\phi_k.
\end{align*}
The norm of the third term is upper-bounded in \eqref{eq:uppboundnorm}. Further, using Bessel's inequality, \eqref{eq:compareeigen}, and Weyl's inequality the norm of the sum over $j\neq k$ is upper-bounded by
\begin{align*}
& \, \frac{4 }{3\, \delta_k}\, \| T_\Delta (\phi_{k,G}-\phi_k)\|_{L_2} + \frac{4 \, \|\Delta\|}{3\, \delta_k^2}\,\| T_\Delta \phi_k\|_{L_2}\\
\leq & \, \frac{4 }{3\, \delta_k}\, \| \Delta\| \| \phi_{k,G}-\phi_k\|_{L_2} + \frac{4 \, \|\Delta\|^2}{3 \, \delta_k^2} \leq  \, \frac{8 \, \sqrt2}{3\, \delta_k^2}\, \| \Delta\|^2  + \frac{4 \, \|\Delta\|^2}{3 \, \delta_k^2}
\end{align*}
upon using \eqref{eq:spectraleigenfct} in the last step. Overall this proves \eqref{eq:expeigfct}.
\end{proof}
\end{revblock}

\begin{revblock}
\begin{proof}[Proof of Theorem \ref{theorem:optimalityeigenval}, general design]
We show that \eqref{eq:lowerboundeigenvalue} remains valid without the additional assumption of a regular design used in the proof in Section \ref{ssec:proof:optimality}.
Set $x_0=0$, $x_{p+1}=1$ and $w_p = (2\,(p+1))^{-1}$. For $l=0,\ldots,p$ let $N_l\in\N_0$ be the largest integer with $N_l\,w_p\leq x_{l+1}-x_l$, choose $N_l$ subintervals of $[x_l,x_{l+1}]$ of length $w_p$ with disjoint interiors, and let $M_p$ be the set of the midpoints of all these subintervals. Now $(x_{l+1}-x_l)/w_p - 1 < N_l \leq (x_{l+1}-x_l)/w_p$. By summing up these inequalities and noting that $\# M_p = N_0+\ldots+N_p$ we obtain
$$ \frac{1}{2\, w_p} = p+1 =  \frac{1}{w_p} - (p+1) < \# M_p \leq \frac{1}{w_p} = 2\,(p+1).$$
Instead of \eqref{eq:testfct} we now set
$$\varphi(x;p) = \sum_{m\in M_p} p^{-\gamma}\,\varphi\big(2\, w_p^{-1}\, (x-m)\big).$$
Then, as in \eqref{eq:testfct}, at each point at most one summand is non-zero, and $\varphi(x_j;p)=0$ at all design points $x_j$. Further,
$$\int\varphi(\cdot;p) = \# M_p\,\frac{p^{-\gamma}\, w_p}{2}\int\varphi \geq \frac{p^{-\gamma}}{4}\int\varphi.$$
Since the summands have supports of length $w_p$ and height of order $p^{-\gamma}$, as before the H\"older norm of $(1+\varphi(x;p))\,(1+\varphi(y;p))$ is bounded uniformly in $p$, and $\|\varphi(\cdot;p)\|_{L_2} = \Oop{p^{-\gamma}}$. The remaining proof of Theorem~\ref{theorem:optimalityeigenval} then applies without further changes.
\end{proof}
\end{revblock}

\begin{revblock}
\begin{proof}[Proof of Theorem \ref{theorem:optimality}, $n^{-1/2}$ and general $k$]

We first prove the lower bound of order $n^{-1/2}$ for general $k$. Then we extend the lower bounds of order $p^{-\gamma}$ and $(\log(n\,p)/(n\,p))^{\gamma/(2\gamma+1)}$ from $k=1$ to general $k$.

\medskip

\textit{Lower bound of order $n^{-1/2}$.} As in \citet[Proposition 2]{belhakem2025minimax} for the first two FPCs and in \citet[Proposition 1]{bourarach2026minimax} for general $k$, the two hypotheses differ by a rotation of two eigenfunctions by an angle of order $n^{-1/2}$. These authors consider the $L_2$-norm and equidistant design points, and evaluate the Kullback-Leibler divergence, respectively the Hellinger distance, between the distributions of the discretized observations explicitly. Here we bound the Kullback-Leibler divergence by the data processing inequality, which gives a short argument for general design points.

Let $\psi_1,\ldots,\psi_k$ be smooth, orthonormal functions on $[0,1]$ with $\int_0^1\psi_j = 0$, which do not depend on $n$ and $p$, and for $\theta\in[0,1]$ set
$$e_\theta = \cos\theta + \sin\theta\,\psi_k,\qquad e_\theta^\perp = -\sin\theta+\cos\theta\,\psi_k,$$
so that $\psi_1,\ldots,\psi_{k-1},e_\theta,e_\theta^\perp$ are orthonormal. With independent $W_{i,j}\sim\mathcal N(0,1)$ consider
$$Z_{i;\theta} = \tilde L\,\Big(\sqrt2\,\sum_{j<k}W_{i,j}\,\psi_j + W_{i,k}\,e_\theta + 2^{-1/2}\,W_{i,k+1}\,e_\theta^\perp\Big),$$
with a constant $\tilde L>0$ as in the proof of Theorem~\ref{theorem:optimalityeigenval}. The covariance kernel of $Z_{i;\theta}$ has the eigenvalues $2\,\tilde L^2$ (with multiplicity $k-1$), $\tilde L^2$, $\tilde L^2/2$ and $0$, where $\psi_1,\ldots,\psi_{k-1}$, $e_\theta$ and $e_\theta^\perp$ are eigenfunctions for the positive eigenvalues. Hence its $k$-th eigenfunction is $e_\theta$, its $k$-th spectral gap is $\tilde L^2/2$, and its H\"older norm is bounded by $C\,\tilde L^2$ with a constant $C$ independent of $\theta$. Choosing $\tilde L^2 = L/C$ we obtain that $Z_{i;\theta}\in\mc P_{\mathcal G}(\gamma;k,L)$ for all $\theta\in[0,1]$ if $\delta\leq L/(2\,C)$. We take the two hypotheses $\theta_0 = 0$ and $\theta_1 = n^{-1/2}$.

\smallskip

\textit{Distance of the hypotheses.} Since the supremum norm on $[0,1]$ dominates the $L_2$-norm and $\langle e_\theta,e_0\rangle = \cos\theta$, for $\theta\in[0,1]$ we have that
$$\dd(e_\theta,e_0)_{\infty,s}\geq\min\big(\|e_\theta-e_0\|_{L_2},\|e_\theta+e_0\|_{L_2}\big) = \min\big(2\,\sin(\theta/2),\,2\,\cos(\theta/2)\big) = 2\,\sin(\theta/2)\geq\frac{2\,\theta}{\pi}.$$
Hence $\dd(e_{\theta_1},e_{\theta_0})_{\infty,s}\geq 2\,n^{-1/2}/\pi$.

\smallskip

\textit{Kullback-Leibler divergence.} Let $\mathrm{KL}(P,Q)$ denote the Kullback-Leibler divergence of the probability measures $P$ and $Q$, and for $\theta\in[0,1]$ let $P_\theta$ be the distribution of the vector $(Y_{1,1},\ldots,Y_{1,p})$ of the observations of the first curve in model \eqref{eq:model} with $\mu=0$ and $Z_1 = Z_{1;\theta}$. Under the hypothesis $\theta$, the vectors $(Y_{i,1},\ldots,Y_{i,p})$, $i=1,\ldots,n$, are independent and identically distributed. Thus the observations have the distribution $P_\theta^{\otimes n}$, and
\begin{equation}\label{eq:KLproduct}
\mathrm{KL}\big(P_{\theta_1}^{\otimes n},P_{\theta_0}^{\otimes n}\big) = n\,\mathrm{KL}\big(P_{\theta_1},P_{\theta_0}\big).
\end{equation}
To bound the divergence for a single curve, write
$$Z_{1;\theta} = \tilde L\,\sqrt2\,\sum_{j<k}W_{1,j}\,\psi_j + \xi_{\theta,1} + \xi_{\theta,2}\,\psi_k,\qquad \begin{pmatrix}\xi_{\theta,1}\\ \xi_{\theta,2}\end{pmatrix} = \tilde L\,R_\theta\begin{pmatrix}W_{1,k}\\ 2^{-1/2}\,W_{1,k+1}\end{pmatrix},$$
where $R_\theta$
is the rotation by the angle $\theta$. Then $\xi_\theta = (\xi_{\theta,1},\xi_{\theta,2})^\top\sim\mathcal N(0,\Sigma_\theta)$ with $\Sigma_\theta = \tilde L^2\,R_\theta\,\mathrm{diag}(1,1/2)\,R_\theta^\top$, and
$$Y_{1,j} = \tilde L\,\sqrt2\,\sum_{l<k}W_{1,l}\,\psi_l(x_j) + \xi_{\theta,1} + \xi_{\theta,2}\,\psi_k(x_j) + \e_{1,j},\qquad j=1,\ldots,p.$$
Hence $(Y_{1,1},\ldots,Y_{1,p})$ is the image of the two vectors $\xi_\theta$ and $(W_{1,1},\ldots,W_{1,k-1},\e_{1,1},\ldots,\e_{1,p})$ under a map which does not depend on $\theta$. The second vector is independent of $\xi_\theta$, and its distribution does not depend on $\theta$. By the data processing inequality, the Kullback-Leibler divergence between $P_{\theta_1}$ and $P_{\theta_0}$ is at most that between the joint distributions of the two vectors under $\theta_1$ and $\theta_0$, which equals that between the distributions of $\xi_{\theta_1}$ and $\xi_{\theta_0}$, that is,
\begin{equation}\label{eq:KLdataprocessing}
\mathrm{KL}\big(P_{\theta_1},P_{\theta_0}\big)\leq\mathrm{KL}\big(\mathcal N(0,\Sigma_{\theta_1}),\mathcal N(0,\Sigma_{\theta_0})\big).
\end{equation}
Finally, since $\det\Sigma_\theta = \det\Sigma_0$ and $\mathrm{tr}\big(\Sigma_0^{-1}\Sigma_\theta\big) =  2+\sin^2(\theta)/2$,
$$\mathrm{KL}\big(\mathcal N(0,\Sigma_\theta),\mathcal N(0,\Sigma_0)\big) = \frac12\,\Big(\mathrm{tr}\big(\Sigma_0^{-1}\Sigma_\theta\big) - 2\Big)  = \frac{\sin^2\theta}4.$$
Together with \eqref{eq:KLproduct} and \eqref{eq:KLdataprocessing}, the Kullback-Leibler divergence between the distributions of the observations under $\theta_1$ and $\theta_0$ is at most $n\,\sin^2(n^{-1/2})/4\leq 1/4$, and the reduction to two hypotheses \citep[Theorem 2.2]{tsybakov2008introduction} yields the lower bound of order $n^{-1/2}$.

\smallskip

\medskip

\textit{Lower bounds of order $p^{-\gamma}$ and $(\log(n\,p)/(n\,p))^{\gamma/(2\gamma+1)}$ for general $k$.}

\smallskip

Before the scaling by a constant factor, each hypothesis in the proof for $k=1$ is of the form $Z_i = \tilde L\, W_i\,u$, where either $u=1$, or $u=1+g$ with $g\in C[0,1]$ such that $\|u\|_{L_2}=1$. Let $\psi_1,\ldots,\psi_{k-1}$ be as in the previous proofs for general $k$, where we choose these functions independent of $n$ and $p$. Then we replace $Z_i = \tilde L\, W_i\,u$ by
$$Z_{i;k} = \tilde L\,\Big(\sqrt2\,\big(W_{i,1}\,\psi_1+\ldots+W_{i,k-1}\,\psi_{k-1}\big) + W_{i,k}\,u\Big),$$
with independent $W_{i,j}\sim\mathcal N(0,1)$ and the scaling factor $\tilde L>0$. With the tensor product $(f\otimes h)(x,y) = f(x)\,h(y)$, the covariance kernel of $Z_{i;k}$ is $\tilde L^2\,G_u$, where
\begin{equation}\label{eq:kernelGu}
G_u = 2\sum_{j=1}^{k-1}\,\psi_j\otimes\psi_j + u\otimes u.
\end{equation}
Let $\lambda_{k;u} = \lambda_{k,G_u}$ denote the $k$-th eigenvalue of $T_{G_u}$ and $\phi_{k;u} = \phi_{k,G_u}$ an associated normalized eigenfunction, which does not depend on $\tilde L$. For $u=1$ we have $\lambda_{k;1}=1$ with spectral gap $1$, and $\phi_{k;1}=1$. For $u=1+g$, however, $u$ is in general not an eigenfunction of $T_{G_u}$, since $g$ need not be orthogonal to $\psi_1,\ldots,\psi_{k-1}$.

\smallskip

We now proceed in two main steps.

\smallskip

\textit{Step 1: Supremum norm bound for the $k$-th eigenfunctions.}
The main step of the proof is to show that $\phi_{k;u}$ is still close to $u$ in the supremum norm: if $\|g\|_{L_2}\leq 1/12$, then $\lambda_{k;u}$ is simple with spectral gap at least $1/2$, and if the sign of $\phi_{k;u}$ is chosen such that $\langle\phi_{k;u},u\rangle\geq 0$, then setting
$$r = 2\sum_{j=1}^{k-1}\langle\psi_j,g\rangle\,\psi_j, \qquad C_\psi = 2\sum_{j=1}^{k-1}\|\psi_j\|_\infty^2$$ we get
\begin{equation}\label{eq:supnormKu}
\|\phi_{k;u} - u\|_\infty\leq\big(1+6\,\|G_u\|_\infty\big)\,\|r\|_\infty\leq C_\psi\,\big(1+6\,\|G_u\|_\infty\big)\int_0^1|g|.
\end{equation}
 If  $g$ in $u=1+g$ is a bump function of height $\eta$ which is supported in an interval of length $2\,\rho$, the right-hand side of \eqref{eq:supnormKu} is of order $\eta\,\rho$, and hence of smaller order than $\|g \|_\infty$ of order $\eta$, which is the distance between $u$ and the constant $1$ in the other hypotheses. Therefore the separation of the hypotheses in $\dd(\cdot,\cdot)_{\infty,s}$ carries over from the functions $u$ to the $k$-th eigenfunctions $\phi_{k;u}$.

To prove \eqref{eq:supnormKu}, as for the upper bounds in \eqref{eq:basicexpansion} we start from bounds in $L_2$ and then pass to the supremum norm by using the eigen-equation. Apply Lemma~\ref{lem:rankkperturbation} below with $\phi_j=\psi_j$ and $\lambda_j=2$ for $j<k$, and with $\phi_k=1$ and $\lambda_k=1$, so that $\delta_k=1$, $G_0 = G_1$, $G=G_u$, and $r$ in \eqref{eq:rankkresidual} is as above. Since $\lambda_1\,\|g\|_{L_2}\leq 1/6$, the lemma shows that $\lambda_{k;u}\geq 3/4$ is simple with spectral gap at least $1/2$, and that
\begin{equation}\label{eq:L2Ku}
\|\phi_{k;u} - u\|_{L_2}\leq 2\,\|r\|_{L_2},\qquad |\lambda_{k;u}-1|\leq 2\,\|r\|_{L_2}.
\end{equation}
Since $\langle\psi_j,1\rangle=0$ and $\|u\|_{L_2}=1$ we have $T_{G_u}u = u+r$, that is, $u = T_{G_u}u-r$. Therefore
$$\phi_{k;u} - u = \lambda_{k;u}^{-1}\,T_{G_u}(\phi_{k;u} - u) + (\lambda_{k;u}^{-1} - 1)\,T_{G_u}u + r.$$
From \eqref{eq:operatornormmixed} we have $\|T_{G_u}f\|_\infty\leq\|G_u\|_\infty\,\|f\|_{L_2}$ for $f\in L_2[0,1]$. Using $\lambda_{k;u}\geq 3/4$, $\|u\|_{L_2}=1$, \eqref{eq:L2Ku} and $\|r\|_{L_2}\leq\|r\|_\infty$, we get
$$\|\phi_{k;u} - u\|_\infty\leq\frac43\,\|G_u\|_\infty\,\big(2+2\big)\,\|r\|_{L_2}+\|r\|_\infty\leq\big(1+6\,\|G_u\|_\infty\big)\,\|r\|_\infty.$$
Finally,
$$\|r\|_\infty\leq 2\sum_{j=1}^{k-1}\|\psi_j\|_\infty\,|\langle\psi_j,g\rangle|\leq C_\psi\int_0^1|g|,$$ which proves \eqref{eq:supnormKu}.

\smallskip

\textit{Step 2: Application to the two constructions.} In the proof for $k=1$,
\begin{equation}\label{eq:singlebump}
g(x) = \eta\,\varphi_w\big((x-t_0)/\rho\big), \qquad w = w(\eta/2),
\end{equation}
with support in $[t_0-\rho,t_0+\rho]\subseteq[0,1]$, where $\varphi_w$ and $w(\cdot)$ are as in Lemma~\ref{lem:constrfct}, and the two constructions are as follows.
\begin{itemize}
\item[(a)] For the rate $p^{-\gamma}$ there are two hypotheses, $u=1$ and $u=1+g$, where in \eqref{eq:singlebump} we have $\eta = p^{-\gamma}$, $\rho = (2\,(p+1))^{-1}$ and $t_0 = (x_l+x_{l+1})/2$, with the index $l$ chosen in the proof for $k=1$.
\item[(b)] For the rate $(\log(n\,p)/(n\,p))^{\gamma/(2\gamma+1)}$ there are $N_{n,p}+1$ hypotheses, $u_0=1$ and $u_l = 1+g_l$, $l=1,\ldots,N_{n,p}$, where $g_l$ is as in \eqref{eq:singlebump} with $\eta = \kappa\,(h_{n,p}/2)^{\gamma}$, $\rho = h_{n,p}/2$ and $t_0 = z_l$.
\end{itemize}
In both cases $\eta\to0$ and $\rho\to0$. Since $\varphi_w$ is supported in $[-1,1]$ and $\|\varphi_w\|_\infty\leq\|\varphi\|_\infty+\|\varphi'\|_\infty$ for all $w\in[0,1]$, setting $c_\varphi = 2\,(\|\varphi\|_\infty+\|\varphi'\|_\infty)$ we have that
\begin{equation}\label{eq:bumpnorms}
\|g\|_\infty\leq c_\varphi\,\eta,\qquad \|g\|_{L_2}\leq c_\varphi\,\eta\,\rho^{1/2},\qquad \int_0^1|g|\leq c_\varphi\,\eta\,\rho.
\end{equation}
Hence for large $p$, respectively large $n\,p$, we have $\|g\|_{L_2}\leq 1/12$ and $\|G_u\|_\infty$ is bounded, so that by Step 1 the kernels $G_u$ have $k$-th spectral gap at least $1/2$, and $\|\phi_{k;u}-u\|_\infty\leq C\,\eta\,\rho$ with a constant $C$ which does not depend on $n$, $p$ and the location $t_0$. Further, since $\eta\leq C\,\rho^\gamma$ the H\"older norms of the kernels $G_u$ are bounded uniformly in $n$ and $p$, so that all hypotheses belong to $\mc P_{\mathcal G}(\gamma;k,L)$ for sufficiently small $\tilde L$ and $\delta$. Finally, under each hypothesis the distribution of the observations is the convolution of that for $k=1$ with the same Gaussian distribution. Hence the observations still have the same distribution under both hypotheses in (a), in (b) the bounds on the Kullback-Leibler divergences remain valid by the data processing inequality, and the arguments for $k=1$ apply.
\end{proof}
\end{revblock}

\begin{revblock}
In the proof of Theorem~\ref{theorem:optimality} for general $k$ we used the following perturbation bound for covariance kernels of finite rank, in which the eigenfunction for the $k$-th eigenvalue is perturbed.

\begin{lemma}\label{lem:rankkperturbation}
Let $\phi_1,\ldots,\phi_k\in L_2[0,1]$ be orthonormal, let $\lambda_1\geq\ldots\geq\lambda_{k-1}>\lambda_k>0$ and set $\delta_k = \min(\lambda_{k-1}-\lambda_k,\lambda_k)$, where $\delta_1 = \lambda_1$. Further, let $g\in L_2[0,1]$ be such that $u = \phi_k+g$ satisfies $\|u\|_{L_2}=1$. Writing $(f\otimes h)(x,y) = f(x)\,h(y)$, consider the covariance kernels
\begin{equation}\label{eq:rankkkernels}
G_0 = \sum_{j=1}^{k}\lambda_j\,\phi_j\otimes\phi_j \qquad\text{and}\qquad G = \sum_{j=1}^{k-1}\lambda_j\,\phi_j\otimes\phi_j+\lambda_k\,u\otimes u,
\end{equation}
so that $\lambda_k$ is a simple eigenvalue of $T_{G_0}$ with eigenfunction $\phi_k$ and spectral gap $\delta_k$, and set
\begin{equation}\label{eq:rankkresidual}
r = \sum_{j=1}^{k-1}\lambda_j\,\langle\phi_j,g\rangle\,\phi_j.
\end{equation}
If $\lambda_1\,\|g\|_{L_2}\leq\delta_k/6$, then the $k$-th eigenvalue $\lambda_{k,G}$ of $T_G$ satisfies $\lambda_{k,G}\geq 3\,\lambda_k/4$, and it is simple with spectral gap at least $\delta_k/2$. Moreover, if the sign of the normalized eigenfunction $\phi_{k,G}$ of $T_G$ for $\lambda_{k,G}$ is chosen such that $\langle\phi_{k,G},u\rangle\geq 0$, then
\begin{align}
\|\phi_{k,G}-u\|_{L_2} & \leq \frac{2}{\delta_k}\,\|r\|_{L_2}, \label{eq:rankkeigfct}\\
|\lambda_{k,G}-\lambda_k| & \leq 2\,\|r\|_{L_2}. \label{eq:rankkeigval}
\end{align}
\end{lemma}

The function $r$ in \eqref{eq:rankkresidual} is the residual in the eigen-equation for $u$, see \eqref{eq:rankkapprox} below. If $g$ is orthogonal to $\phi_1,\ldots,\phi_{k-1}$, then $r=0$ and $u$ is the eigenfunction of $T_G$ for the eigenvalue $\lambda_{k,G}=\lambda_k$. Thus \eqref{eq:rankkeigfct} and \eqref{eq:rankkeigval} are bounds in terms of the residual for the approximate eigenfunction $u$ and the approximate eigenvalue $\lambda_k$ of $T_G$. Bounds of this type are classical, see e.g.~\citet[Section 11.7]{parlett1998symmetric}. We provide the proof for convenience.

\begin{proof}[Proof of Lemma~\ref{lem:rankkperturbation}]
\textit{Spectrum of $T_G$.} Since $\|u\|_{L_2}=\|\phi_k\|_{L_2}=1$ we have that $|\langle u,\phi_k\rangle|\leq 1$ and $\|g\|_{L_2}^2 = 2-2\,\langle u,\phi_k\rangle$, hence
\begin{align*}
\|G-G_0\|_{L_2([0,1]^2)}^2 & = \lambda_k^2\,\|u\otimes u-\phi_k\otimes\phi_k\|_{L_2([0,1]^2)}^2 = 2\,\lambda_k^2\,\big(1-\langle u,\phi_k\rangle^2\big)\\
& \leq 4\,\lambda_k^2\,\big(1-\langle u,\phi_k\rangle\big) = 2\,\lambda_k^2\,\|g\|_{L_2}^2,
\end{align*}
so that $\|G-G_0\|_{L_2([0,1]^2)}\leq\sqrt2\,\lambda_1\,\|g\|_{L_2}\leq\delta_k/4$. The eigenvalues of $T_{G_0}$ are $\lambda_1,\ldots,\lambda_k$ and $\lambda_j = 0$ for $j>k$, so that $|\lambda_j-\lambda_k|\geq\delta_k$ for $j\neq k$. By \eqref{eq:l2operator} and Weyl's inequality \eqref{eq:spectraleigenvalues}, the eigenvalues $\lambda_{1,G}\geq\lambda_{2,G}\geq\ldots\geq 0$ of $T_G$ satisfy $|\lambda_{j,G}-\lambda_j|\leq\delta_k/4$ for all $j$. Since $\delta_k\leq\lambda_k$ this implies that $\lambda_{k,G}\geq\lambda_k-\delta_k/4\geq 3\,\lambda_k/4$, that $|\lambda_{k,G}-\lambda_{j,G}|\geq\delta_k/2$ for $j\neq k$, and that
\begin{equation}\label{eq:rankkaway}
|\lambda_{j,G}-\lambda_k|\geq 3\,\delta_k/4,\qquad j\neq k.
\end{equation}

\textit{Approximate eigen-equation for $u$.} Since $\langle\phi_j,\phi_k\rangle=0$ for $j<k$ and $\langle u,u\rangle=1$,
\begin{equation}\label{eq:rankkapprox}
T_Gu = \sum_{j=1}^{k-1}\lambda_j\,\langle\phi_j,u\rangle\,\phi_j+\lambda_k\,\langle u,u\rangle\,u = \lambda_k\,u+r.
\end{equation}

\textit{Proof of \eqref{eq:rankkeigfct}.} Let $Qf = f-\langle f,\phi_{k,G}\rangle\,\phi_{k,G}$ be the orthogonal projection onto the orthogonal complement of $\phi_{k,G}$, and set $a = \langle u,\phi_{k,G}\rangle\geq 0$, so that $u = a\,\phi_{k,G}+Qu$. Since $T_G$ is self-adjoint and $T_G\phi_{k,G}=\lambda_{k,G}\,\phi_{k,G}$, $Q$ commutes with $T_G$, and \eqref{eq:rankkapprox} gives $(T_G-\lambda_k\,\mathrm{id})\,Qu = Q\,(T_G-\lambda_k\,\mathrm{id})\,u = Qr$. Expanding $Qu$ in an orthonormal basis $(\phi_{j,G})_{j\neq k}$ of the orthogonal complement of $\phi_{k,G}$ which consists of eigenfunctions of $T_G$ for the eigenvalues $\lambda_{j,G}$, from \eqref{eq:rankkaway} we obtain
$$\|Qr\|_{L_2}^2 = \big\|(T_G-\lambda_k\,\mathrm{id})\,Qu\big\|_{L_2}^2 = \sum_{j\neq k}(\lambda_{j,G}-\lambda_k)^2\,\langle Qu,\phi_{j,G}\rangle^2\geq\frac{9\,\delta_k^2}{16}\,\|Qu\|_{L_2}^2,$$
so that
\begin{equation}\label{eq:rankkQu}
\|Qu\|_{L_2}\leq\frac{4}{3\,\delta_k}\,\|Qr\|_{L_2}\leq\frac{4}{3\,\delta_k}\,\|r\|_{L_2}.
\end{equation}
Further, $\|u\|_{L_2}=\|\phi_{k,G}\|_{L_2}=1$ implies that $a^2 = 1-\|Qu\|_{L_2}^2$ and, using $0\leq a\leq 1$, that
$$\|\phi_{k,G}-u\|_{L_2}^2 = 2-2\,a = 2\,(1-a)\leq 2\,(1-a)\,(1+a) = 2\,\|Qu\|_{L_2}^2.$$
Together with \eqref{eq:rankkQu} we obtain
$$\|\phi_{k,G}-u\|_{L_2}\leq\sqrt2\,\|Qu\|_{L_2}\leq\frac{4\sqrt2}{3\,\delta_k}\,\|r\|_{L_2},$$
which implies \eqref{eq:rankkeigfct}.

\textit{Proof of \eqref{eq:rankkeigval}.} Since $T_G$ is self-adjoint, $\langle T_Gu,\phi_{k,G}\rangle = \langle u,T_G\phi_{k,G}\rangle = \lambda_{k,G}\,a$. Hence taking the inner product of \eqref{eq:rankkapprox} with $\phi_{k,G}$ gives $(\lambda_{k,G}-\lambda_k)\,a = \langle r,\phi_{k,G}\rangle$. By Bessel's inequality, $\|r\|_{L_2}\leq\lambda_1\,\|g\|_{L_2}\leq\delta_k/6$, so that \eqref{eq:rankkQu} yields $a^2 = 1-\|Qu\|_{L_2}^2\geq 1-4/81$ and hence $a\geq 1/2$. With the Cauchy--Schwarz inequality we obtain
$$|\lambda_{k,G}-\lambda_k| = \frac{|\langle r,\phi_{k,G}\rangle|}{\langle u,\phi_{k,G}\rangle}\leq 2\,\|r\|_{L_2}.$$
\end{proof}
\end{revblock}

\section{Results on covariance kernel estimation}\label{sec:appendix1}

\rev{The following assumption on the weights of the linear estimator in \eqref{eqn:estimatorCovariance} and the rates of convergence of the covariance kernel estimator in Lemma~\ref{thm:rates_covkernrates} are taken from \citet{berger2024optimal}.}

\begin{assumption}\label{ass:weights}
	There is a $c>0$ and a $h_0>0$ such that for  sufficiently large $p$, for the weights $\wjk xyh$ the following holds for all $h \in (c/p\rev{,} h_0]$ for constants $\Cmax, \Clip>0$ which are independent of $n, p,h$ and $(x, y) \in T$.
	\begin{enumerate}[label=\normalfont{(W\arabic*)},leftmargin=9.9mm]
		\item The weights reproduce polynomials of a degree $\zeta \geq 0$, that is for $(x,y)  \in T$,
		\begin{align*}
			\sum_{ j<k}^{ p}  \wjk xyh =1\,, \quad \sum_{j<k}^{ p}( x_{ j}- x)^{r_1}(x_k - y)^{r_2} \,  \wjk xyh = 0\,,
		\end{align*} \label{ass:weights:polynom}
        for $r_1,r_2 \in \N_0$ s.t. {\color{black} $1 \leq $} $r_1 + r_2 \leq \zeta$.
		\item We have $\wjk xyh = 0$ if $\max(\abs{x_j-x}, \abs{x_k - y})> h$ with $(x,y) \in T$. \label{ass:weights:vanish}
		\item For the absolute values of the weights it holds $  \max_{j<k} \big| \wjk xyh\big|  \leq \Cmax (p\,h)^{-2},\,  x \leq y.$  \label{ass:weights:sup}
		\item For a Lipschitz constant $\Clip > 0$ it holds that
		\begin{align*}
			\absb{  \wjk xyh -  \wjk{x^\prime}{y^\prime}h} \leq \frac{\Clip}{( p\, h)^2} \bigg(\frac{\max(\abs{x-x^\prime}, \rev{\abs{y-y^\prime}})}h \wedge 1\bigg), \ \;x\leq y,\, x^\prime\leq y^\prime \,.
		\end{align*} \label{ass:weights:lipschitz}

	\end{enumerate}
\end{assumption}

\begin{lemma}[Rates for covariance estimation]\label{thm:rates_covkernrates}
	\rev{Consider model \eqref{eq:model} under Assumption~\ref{ass:design:localization}, let $0<\beta_0\leq1$ and $L, C_Z>0$, and suppose that the weights of $\hat\Gamma_n(\cdot;h)$ in \eqref{eqn:estimatorCovariance} satisfy Assumption~\ref{ass:weights}, (W1) with $\zeta = \lfloor\gamma\rfloor$, (W2) and (W3).}
    \begin{enumerate}
     \item Then under Assumption \ref{ass:distribution}, (i) on the errors,
 \begin{align*}
		\sup_{h \in (c/p, h_0]}\, \sup_{Z \in \mc P (\gamma;L,\beta_0, C_Z)} \bar a_{n,p,h}^{-2}\,  \expec\big[  \big\|\hat \Gamma_n(\cdot;h) - \Gamma \big\|_{L_2}^2\big] = \mc O\big(1\big)\,,
	\end{align*}
	where
	\begin{align}\label{eq:upperboundwithbandwitdh1}
		\bar a_{n,p,h} & = \max \Big( h^\gamma, \Big( \frac{1}{n\,p\,h}\Big)^{1/2}, n^{-1/2}\Big)\,.
	\end{align}
	 Hence setting
	$
		\bar h^\star \sim \max \Big( c/p, \Big(\frac{1}{n\,p}\Big)^{\frac1{2\gamma + 1}} \Big)\,
	$
	we obtain that
     \begin{align}\label{eq:covrateLtwo}
		 \sup_{Z \in \mc P (\gamma)} \,  \expec\big[  \big\|\hat \Gamma_n(\cdot;\bar h^\star) - \Gamma \big\|_{L_2}^2\big] = \mc O\Big(\rev{\max\Big(p^{-2\gamma},  n^{-1} \Big)}\,\Big)\,.
	\end{align}
       \item If Assumption \ref{ass:distribution}, (i) and (ii) and additionally Assumption \ref{ass:weights}, (W4) holds, then
 \begin{align*}
		\sup_{h \in (c/p, h_0]}\, \sup_{Z \in \mc P (\gamma;L,\beta_0, C_Z)}  a_{n,p,h}^{-1}\,  \expec\big[  \big\|\hat \Gamma_n(\cdot;h) - \Gamma \big\|_{\infty}\big] = \mc O\big(1\big)\,,
	\end{align*}
	where
	\begin{align}\label{eq:upperboundwithbandwitdh}
		a_{n,p,h} & = \max \Big( h^\gamma, \Big( \frac{\log(h^{-1})}{n\,p\,h}\Big)^{1/2}, n^{-1/2}\Big)\,.
	\end{align}
	Hence
     setting
	$	h^\star \sim \max \Big( c/p, \Big(\frac{\log(n\,p)}{n\,p}\Big)^{\frac1{2\gamma + 1}} \Big)\,$
	we obtain
     \begin{align}\label{eq:covratesup}
		\sup_{Z \in \mc P (\gamma;L,\beta_0, C_Z)}  \,   \expec\big[  \big\|\hat \Gamma_n(\cdot; h^\star) - \Gamma \big\|_{\infty}\big] = \mc O\Big(\max\Big(p^{-\gamma}, \Big(\frac{\log(n\,p)}{n\,p}\Big)^{\frac \gamma{2\gamma + 1}},  n^{-1/2} \Big)\,\Big)\,.
	\end{align}
    \end{enumerate}
\end{lemma}

\begin{proof}[Proof of Lemma \ref{thm:rates_covkernrates}.]
    Part (ii) is \citet[Theorem 1]{berger2024optimal}. The $L_2$-rate is derived similarly and actually more easily, \rev{with details provided below}.
\end{proof}

For the following Lemma consider the error decomposition for the covariance estimator $\hat \Gamma_n(x,y;h)$ under Assumption \ref{ass:design:localization} in \citet[Lemma 7]{berger2024optimal}. The following Lemma yields upper bounds for the rate of convergence in $L_2$ norm.

\begin{lemma}[Upper bounds in the $L_2$-norm]
	Given  $0 < \beta_0 \leq 1$ and $\gamma, L, C_Z>0$, for the estimator $\covest \cdot\cdot h$ in \eqref{eqn:estimatorCovariance} with $n$ and $p$ large enough the following rates of convergence hold, where we abbreviate $\Zclass = \mc P (\gamma;L,\beta_0, C_Z)$.
	\begin{enumerate}
		\item[i).] If the weights satisfy \ref{ass:weights:polynom} with $\zeta = \floor{\gamma}$, \ref{ass:weights:vanish} and \ref{ass:weights:sup} of Assumption \ref{ass:weights}, then
		\begin{align*}
		\sup_{h \in (\nicefrac cp, h_0]}\, \sup_{Z \in \Zclass }\, h^{-2\gamma} \int_{[0,1]^2} \absb{\sum_{j<k}^p \wjk xyh \big(\Gamma(x_j, x_k) - \Gamma(x,y)\big)}^2 \dx(x,y)= \mc O\big( 1\big)\,,
		\end{align*}
	with constants $h_0, c >0 $ according to Assumption \ref{ass:weights}.
		\item[ii).] If the weights satisfy \ref{ass:weights:vanish}, \ref{ass:weights:sup} and \ref{ass:weights:lipschitz}, then we have that
		\begin{align*}
			\expec\Big[ \int_{T}\absb{ \sum_{j<k}^p \wjk xyh \frac1n \sum_{i=1}^n \epsilon_{i,j} \epsilon_{i,k} }^2\dx(x,y) \Big] &= \mc O \bigg(  {\frac{1}{n\,(p\,h)^2}}\bigg)\,,\\
			\expec\Big[\int_T\absb{ \sum_{j<k}^p \wjk xyh \sum_{i \neq l}^n \frac{\epsilon_{i,j} \epsilon_{l,k}}{n(n-1)} }^2\dx(x,y) \Big] &= \mc O \bigg( \frac{ 1}{(n\,p\,h)^2}\bigg)\,,
		\end{align*}
        where the constants in the $\mc O$ terms can chosen uniformly for $h \in (c/p, h_0]$.
		\item[iii).] If the weights satisfy \ref{ass:weights:vanish} and \ref{ass:weights:sup}, then
		\begin{align*}
			\sup_{h \in (\nicefrac cp, h_0]} \sup_{Z \in \Zclass}  & \expec\Big[ \int_T \absb{ \frac1{n} \sum_{j<k}^p \wjk xyh \sum_{i=1}^n \big(Z_i(x_j)Z_i(x_k) - \Gamma(x_j, x_k)\big) }^2\dx(x,y) \Big] =\mc O\bigg( \frac1n\bigg),\\
			\sup_{h \in (\nicefrac cp, h_0]}\sup_{Z \in \Zclass}  & \expec\Big[ \int_T \absb{\frac1n\sum_{j<k}^p \wjk xyh \sum_{i\neq l=1}^n  \frac{Z_i(x_j)Z_l(x_k)}{n-1}} ^2\dx(x,y)\Big] =\mc O\bigg(\frac1{n^2}\bigg)\,.
		\end{align*}
		\item[iv).] If the weights satisfy \ref{ass:weights:vanish}, \ref{ass:weights:sup} and \ref{ass:weights:lipschitz}, then
		\begin{align*}
			 \sup_{Z \in \Zclass}\expec\Big[ \int_T\absb{ \sum_{j<k}^p \wjk xyh \frac1n \sum_{i=1}^n Z_i(x_j) \epsilon_{i,k} }^2\dx(x,y) \Big] &= \mc O \bigg(\frac{1}{n\,p\,h}\bigg)\,,\\
			 \sup_{Z \in \Zclass}\expec\Big[\int_T \absb{ \sum_{j<k}^p \wjk xyh \sum_{i \neq l}^n \frac{Z_i(x_j) \epsilon_{l,k}}{n(n-1)} }^2\dx(x,y) \Big] &= \mc O \bigg(\frac{1}{n\,p\,h}\bigg)\,,
		\end{align*}
	where the constants in the $\mc O$ terms can chosen uniformly for $h \in (c/p, h_0]$.
	\end{enumerate}
\end{lemma}

\begin{proof}
    \textit{i).} Since the upper bounds in \citet[Lemma 8]{berger2024optimal} hold point-wise for all $(x,y)\in T$ the upper bound follows directly. \\

    \textit{ii).} With the properties of the weights and Fubini for non-negative functions we obtain
    \begin{align*}
        \expec\bigg[ \int_T \bigg( \sum_{j <k}^p \wjk xyh \frac1n \sum_{i = 1}^n \epsilon_{i,j} \epsilon_{i,k} \bigg)^2 \dx (x,y) \bigg]  &= \frac{1}{n^2}\sum_{i = 1}^n \int_{T} \sum_{j <k }^pw_{j,k}^2(x,y;h)  \expec \big[\epsilon_{i,j}^2 \epsilon_{i,k}^2 \big] \dx (x,y)\\
         \leq  \frac1{n} \frac{\Ccard\,\Cmax}{(p\,h)^2}\expec \big[ \epsilon_{1,1}^2\epsilon_{1,2}^2\big]  &= \frac{\Cmax\,\Ccard\,\sigma^4}{n\,(p\,h)^2} = \mc O \bigg( \frac1{n\,(p\,h)^2}\bigg)\,.
    \end{align*}

    \textit{iii).} By the proof of \citet[Lemma 8 iii), (31)]{berger2024optimal} it is sufficient to derive the bound of
    \begin{align*}
        \expec\bigg[ \int_T \bigg( \frac{1}{n}\sum_{i = 1}^n \big(Z_i(x)Z_i(y) - \Gamma(x,y)\big)\rev{\bigg)}^2 \dx (x,y)  \bigg]\,.
    \end{align*}
    Now let $X_{n,i}(x,y) \defeq (Z_i(x)Z_i(y)-\Gamma(x,y))/\sqrt{n}$, which has the envelope $\Phi_{n,i} = (2Z_i^2(0) + 2\rev{M_i^2} + L)/\sqrt n$. Now, again by Fubini and the definition of an envelope inserting the point-wise bound yields
    \begin{align*}
        n\, \expec\bigg[\int_T \Big( \frac{1}{\sqrt n} \sum_{i = 1}^n X_{n,i}(x,y) \Big)^2 \,\dx(x,y)\bigg] & = \sum_{i = 1}^n\int_T  \expec \big[ X_{n,i}^2(x,y) \big]\dx(x,y) \leq \sum_{i = 1}^n \expec\big[ \Phi_{n,i}^2 \big] = \mc O(1)\,.
    \end{align*}

    \textit{iv).} Using Fubini and the properties of the weights
    \begin{align*}
        \expec\bigg[ &\int_T \Big( \sum_{j<k}^p \wjk xyh \frac1n \sum_{i = 1}^n Z_i(x_j)\rev{\epsilon_{i,k}}\Big)^2\dx(x,y) \bigg] \\
        & = \frac1{n^2} \sum_{i = 1}^n \sum_{k = 2}^p \sum_{j,j^\prime = 1}^{k-1} \int \expec \big[ \wjk xyh w_{j^\prime, k}(x,y;h) Z_i(x_j)Z_i(x_{j^\prime}) \epsilon_{i,k}^2 \big]\dx(x,y)  \\
        & \leq \frac1{n^2} \sum_{i = 1}^n \int \expec \big[ \sum_{k = 2}^p \sum_{j = 1}^{k-1}\abs{\wjk xyh} \sum_{j^\prime = 1}^{k-1} \abs{w_{j^\prime, k}(x,y;h)} \abs{Z_i(x_j)}\abs{Z_i(x_{j^\prime})} \epsilon_{i,k}^2 \big]\dx(x,y)  \\
        & \leq \frac{\Cmax^2\,\Ccard^3}{n\,p\,h} \expec\big[(\abs{Z_1(0)} + M_1)^2 \,\epsilon_{1,1}^2 \big] = \mc O \bigg(\frac1{n\,p\,h}\bigg)\,.
    \end{align*}
\end{proof}

\section{Additional numerical results}\label{sec:trapez}

\begin{figure}[h!]
    \centering
    \begin{subfigure}{0.4\linewidth}
    \includegraphics[width=\linewidth]{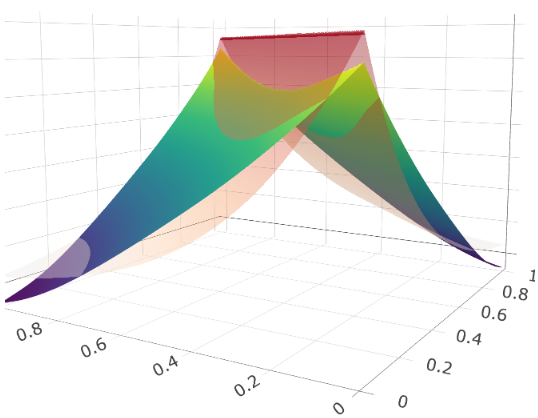}
    \caption{\small \href{https://kevinwilk.github.io/Functional-principal-component-estimation/estimation_covh065.html}{$\hat{\Gamma}_{500,50}(\cdotp,\cdotp;0{.65})$}.}
    \end{subfigure}
    \hfill
    \begin{subfigure}{0.4\linewidth}
    \includegraphics[width=\linewidth]{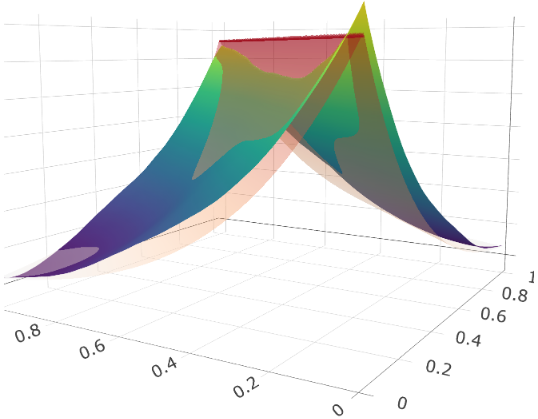}
    \caption{\small \href{https://kevinwilk.github.io/Functional-principal-component-estimation/estimation_covh025.html}{$\hat{\Gamma}_{500,50}(\cdotp,\cdotp;0{.25})$}.}
    \end{subfigure}
    \caption{Covariance estimation (blue–yellow colored area) and covariance kernel (red colored area) of the OU-process \eqref{eq:cov kernel}.}
    \label{fig: est cov}
\end{figure}

\begin{figure}[h!]
    \centering
    \begin{subfigure}{0.49\linewidth}
    \includegraphics[width=\linewidth]{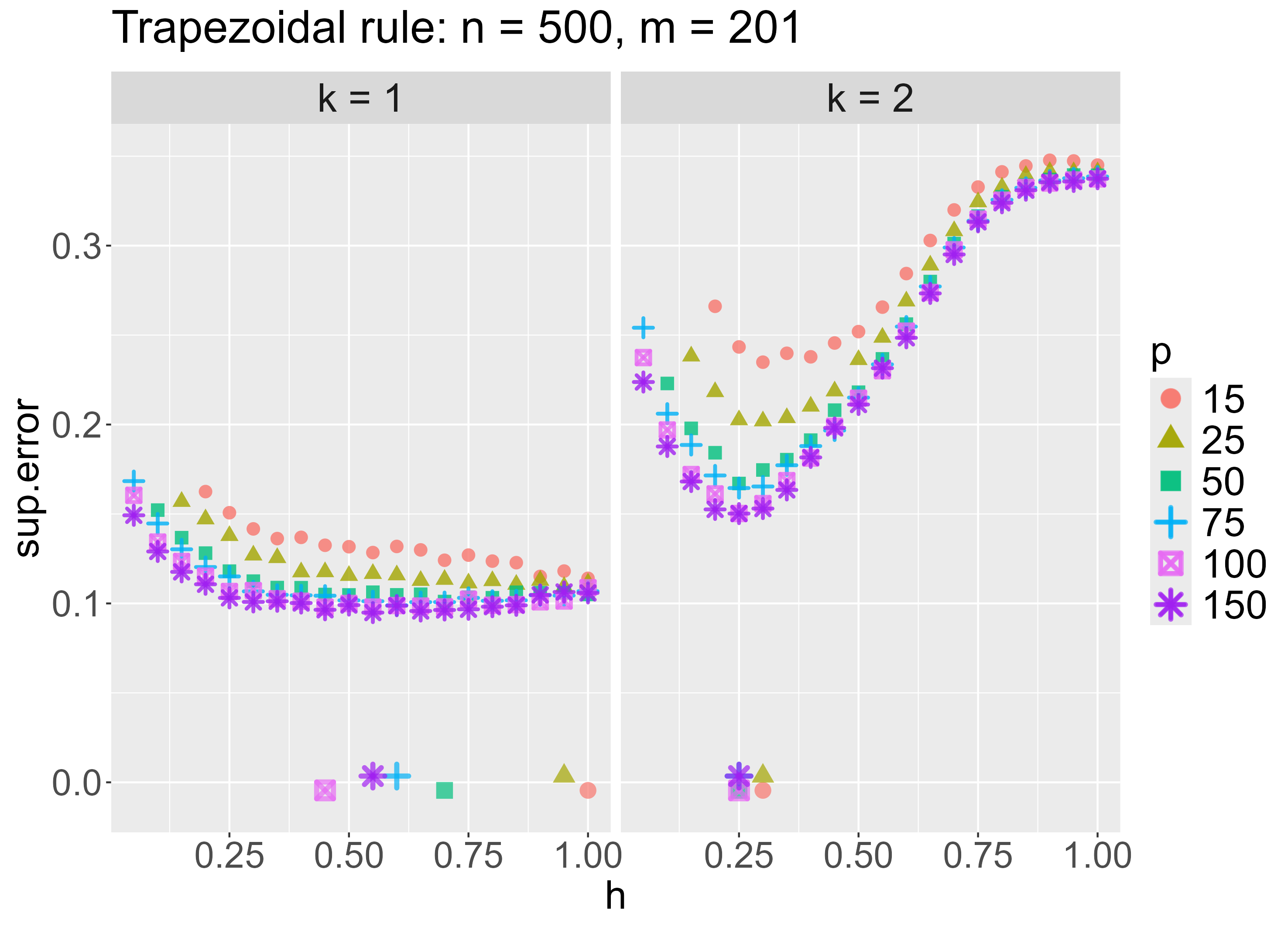}
    \end{subfigure}
    \caption{Analogous to Figure \ref{fig: optimal slection}(a) by using the trapezoidal rule.}
\end{figure}

\begin{figure}[h!]
    \centering
    \begin{subfigure}{0.49\linewidth}
    \includegraphics[width=\linewidth]{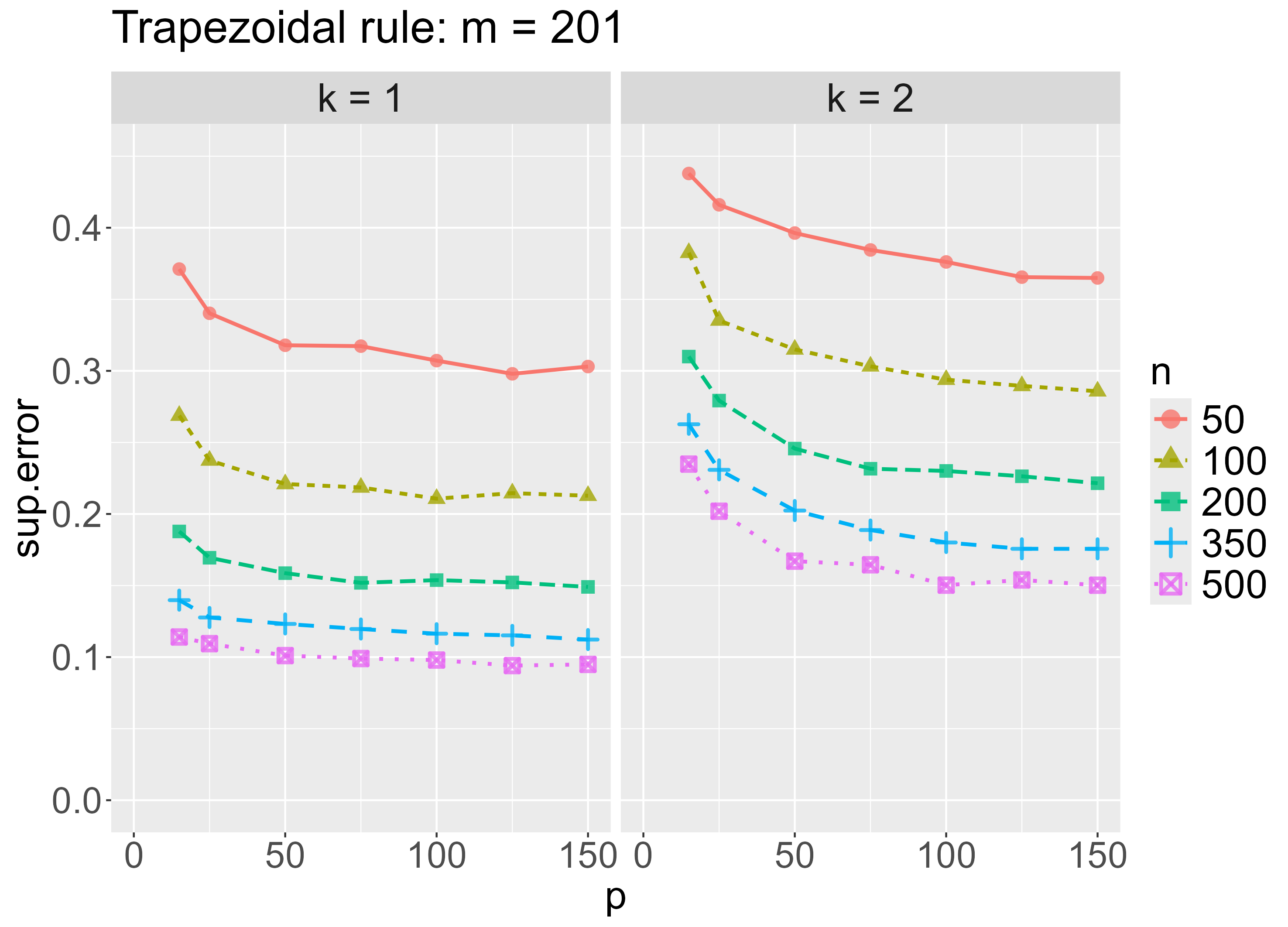}
    \caption{\small Mean of $\|\hat{\phi}_{k;n,201}(\cdotp;h^*) - \hat{c}_{k;201} \, \phi_k \|_{\infty}$.}
    \end{subfigure}
    \hfill
    \begin{subfigure}{0.49\linewidth}
    \includegraphics[width=\linewidth]{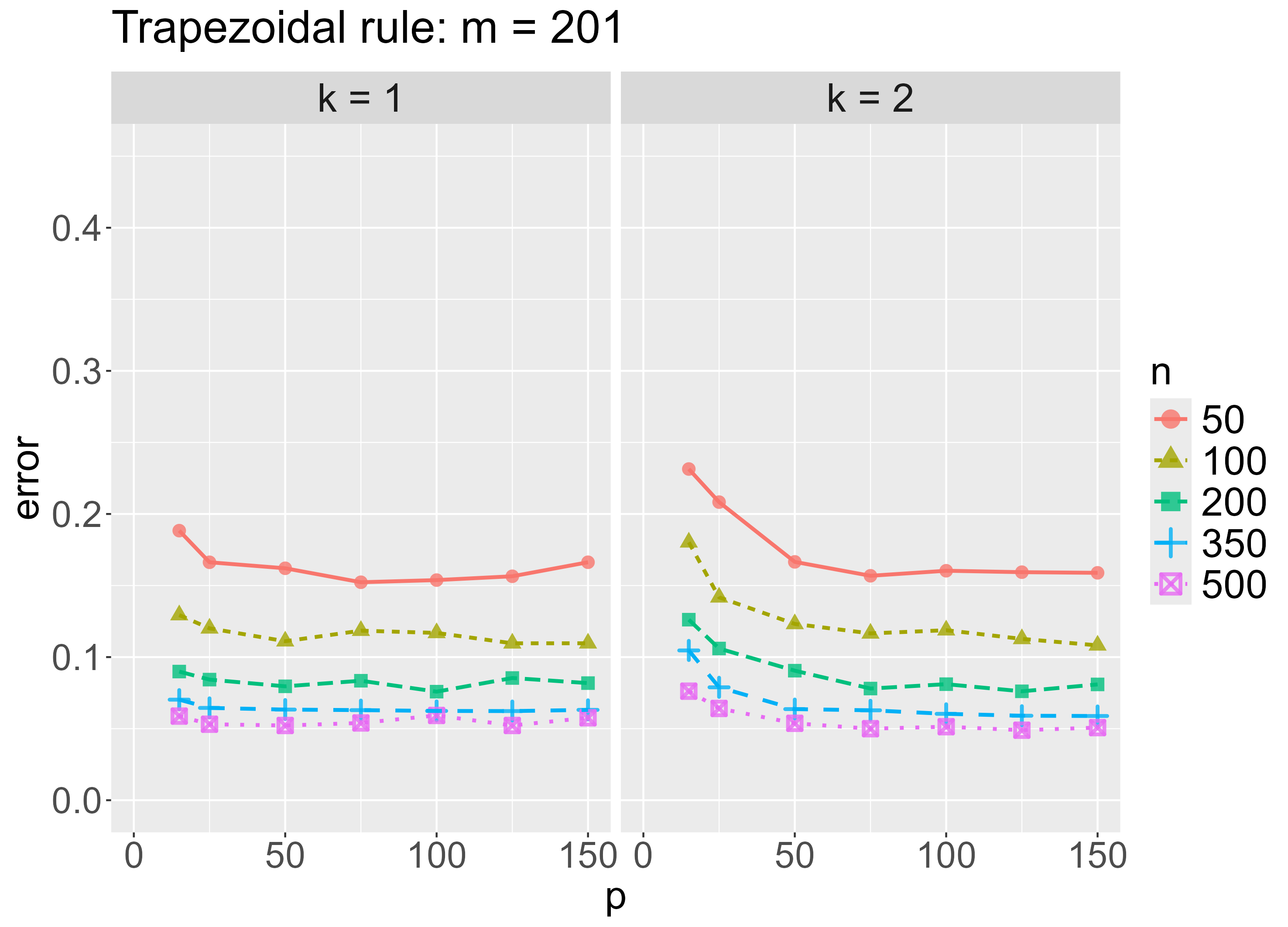}
    \caption{\small Mean of $|(\hat{\lambda}_{k;n,201}(h^*) - \lambda_{k})/\lambda_{k}|$.}
    \end{subfigure}
    \caption{Analogous to Figure \ref{fig: supnorm error} by using the trapezoidal rule.}
\end{figure}

\end{document}